\documentclass[12pt,leqno]{amsart}
\usepackage{latexsym,amsmath,amssymb,xcolor}
\usepackage{graphicx,hyperref}

\title[Factorization through trees]{Lipschitz mappings, metric differentiability, and factorization through metric trees}
\author{Behnam Esmayli, Piotr Haj{\l}asz}

\address{B.\ Esmayli: Department of Mathematics, University of Pittsburgh, Pittsburgh, PA 15260, USA, {\tt bee23@pitt.edu}}

\address{P.\ Haj{\l}asz: Department of Mathematics, University of Pittsburgh, 301
  Thackeray Hall, Pittsburgh, PA 15260, USA, {\tt hajlasz@pitt.edu}}
\thanks{P.H. was supported by NSF grant DMS-2055171.}

plus 6pt minus 12pt
\newtheorem{theorem}{Theorem}
\newtheorem{lemma}[theorem]{Lemma}
\newtheorem{corollary}[theorem]{Corollary}
\newtheorem{proposition}[theorem]{Proposition}

\theoremstyle{definition}
\newtheorem{remark}[theorem]{Remark}

\newtheorem{example}[theorem]{Example}

\newcommand{\barint}{
\rule[.036in]{.12in}{.009in}\kern-.16in \displaystyle\int }

\newcommand{\barcal}{\mbox{$ \rule[.036in]{.11in}{.007in}\kern-.128in\int $}}

\newcommand{\bbbn}{\mathbb N}

\newcommand{\bbbr}{\mathbb R}
\newcommand{\bbbs}{\mathbb S}
\newcommand{\bbbh}{\mathbb H}

\newcommand{\bbbb}{\mathbb B}

\def\diam{\operatorname{diam}}

\def\dist{\operatorname{dist}}

\def\rank{\operatorname{rank}}
\def\md{\operatorname{md}}

\def\H{{\mathcal H}}
\def\id{{\rm id\, }}
\def\lip{{\rm Lip\,}}

\def\eps{{\varepsilon}}

\def\mvint_#1{\mathchoice
          {\mathop{\vrule width 6pt height 3 pt depth -2.5pt
                  \kern -8pt \intop}\nolimits_{\kern -3pt #1}}%
          {\mathop{\vrule width 5pt height 3 pt depth -2.6pt
                  \kern -6pt \intop}\nolimits_{#1}}%
          {\mathop{\vrule width 5pt height 3 pt depth -2.6pt
                  \kern -6pt \intop}\nolimits_{#1}}%
          {\mathop{\vrule width 5pt height 3 pt depth -2.6pt
                  \kern -6pt \intop}\nolimits_{#1}}}

\numberwithin{theorem}{section} \numberwithin{equation}{section}

\begin{document}

\sloppy

\date{\today}

\begin{abstract}
Given a Lipschitz map $f$ from a cube into a metric space, we find several equivalent conditions for $f$ to have a Lipschitz factorization through a metric tree. As an application we prove a recent conjecture of David and Schul. The techniques developed for the proof of the factorization result yield several other new and seemingly unrelated results. We prove that if $f$ is a Lipschitz mapping from an open set in $\mathbb{R}^n$ onto a metric space $X$, then the topological dimension of $X$ equals $n$ if and only if $X$ has positive $n$-dimensional Hausdorff measure. We also prove an area formula for length-preserving maps between metric spaces, which gives, in particular, a new formula for integration on countably rectifiable sets in the Heisenberg group. 
\end{abstract}

\maketitle


\section{Introduction}
\label{Intro}

The purpose of this paper is to discuss factorization of Lipschitz mappings between metric spaces. Given metric spaces $X,Y,Z$, and a Lipschitz map $f:X\to Y$, we say that {\em $f$ factors through $Z$} if there are Lipschitz mappings $\psi:X\to Z$ and $\phi:Z\to Y$ such that $f=\phi\circ\psi$. 

Given a Lipschitz map $f\colon X\to Y$, our aim is to construct a space $Z$ with a simple structure, along with a factorization $f=\phi\circ\psi$. In particular, we are interested in answering the question under what conditions, $f$ factors through a metric tree (see Section~\ref{TR} for the definition of a metric tree).
This question was partially motivated by the recent works of Wenger and Young \cite{wengery} and David and Schul \cite{DS}. 
The next result which is one of the main results of the paper, provides several equivalent conditions for a factorization of a Lipschitz map through a metric tree.

Throughout the paper $n$ and $m$ will stand for nonnegative integers.
\begin{theorem}
\label{TM1}
If $f:Q_o=[0,1]^{n}\to X$, $n\geq 2$, is a Lipschitz map into a metric space, then the following conditions are equivalent:
\begin{itemize}
\item[(a)] $f$ factors through a metric tree.
\item[(b)] $\rank\md(f,x)\leq 1$ almost everywhere.
\item[(c)] $\Theta^{*2}(f,x)=0$ almost everywhere.
\item[(d)] $\Theta^{2}_*(f,x)=0$ almost everywhere. 
\item[(e)] $\H^{2,n-2}_\infty(f,Q_o)=0$.
\item[(f)] $\hat{\H}^{2,n-2}_\infty(f,Q_o)=0$.
\end{itemize}
\end{theorem}
Notation used in Theorem~\ref{TM1} is briefly introduced after the statement of Theorem~\ref{TM1B}, and carefully explained in the subsequent sections.

\begin{remark}
In fact, in (a) we obtain quantitative estimates for the Lipschitz constants. Precisely, if $f$ is $L$-Lipschitz, then we find a metric tree $Z$ and maps $\psi\colon Q_o \to Z$ and $\phi\colon Z \to X$ such that $\psi$ is $L$-Lipschitz, $\phi$ is $1$-Lipschitz and $f=\phi \circ \psi$. The bounds follow from Lemma~\ref{FTT2} and the fact that the cube is $1$-quasiconvex.
\end{remark}

Equivalence of conditions (a) and (e) proves a recent conjecture of David and Schul \cite[Conjecture~1.13]{DS}. They conjectured that if  $f\colon Q_o=[0,1]^3\to X$ satisfies $\H^{2,1}_\infty(f,Q_o)=0$, then $f$ factors through a metric tree.

Recently David and Schul \cite{DS2}, used our result (implication (e) $\Rightarrow$ (a)) to prove a quantitative part of Conjecture~1.13 from \cite{DS} which states that if the content $\H^{2,1}_\infty(f,Q_o)$ is small, then $f$ is close to a mapping $g$ that factors through a tree. 

More precisely, they proved that for every $\eps>0$ there is $\delta=\delta(\eps,m)$, $m\geq 1$, such that if $f:Q_o=[0,1]^{2+m}\to \ell^\infty$ satisfies $\H^{2,m}_\infty(f,Q_o)<\delta$, then $f$ is within distance of $\eps$ to a map $g:Q_o\to\ell^\infty$ that factors through a metric tree. In fact, they proved that $f$ is within distance of $\eps$ to  $g$ such that $\H^{2,m}_\infty(g,Q_o)=0$ and they used 
implication (e)$\Rightarrow$(a) of
Theorem~\ref{TM1} to conclude that $g$ factors through a metric tree.

The equivalence of conditions (b)-(f) in Theorem~\ref{TM1} is a consequence of a more general result:
\begin{theorem}
\label{TM1B}
If $f:[0,1]^{n+m}\to X$, $n\geq 1$, $m\geq 0$, is a Lipschitz map into a metric space, and $E\subset [0,1]^{n+m}$ is a measurable set, then the following conditions are equivalent:
\begin{itemize}
\item[(b')] $\rank\md(f,x)\leq n-1$ almost everywhere in $E$.
\item[(c')] $\Theta^{*n}(f,x)=0$ almost everywhere in $E$.
\item[(d')] $\Theta^{n}_*(f,x)=0$ almost everywhere in $E$. 
\item[(e')] $\H^{n,m}_\infty(f,E)=0$.
\item[(f')] $\hat{\H}^{n,m}_\infty(f,E)=0$.
\end{itemize}
\end{theorem}
Let us now briefly explain notation used above.

A {\em metric tree}, also known as an {\em $\bbbr$-tree}, is a geodesic space which contains no subsets homeomorphic to $\bbbs^1$, so it is a geodesic space without ``loops''. Other equivalent definitions are explained in Section~\ref{TR}.

Kirchheim \cite{kir}, proved that every Lipschitz map $f:\Omega\to X$ from an open set $\Omega\subset\bbbr^n$ into a metric space is {metrically differentiable} a.e. 
With the metric derivative $\md(f,x)$, we can associate its rank (see Section~\ref{KRT} for details).

Following \cite{HZ},
for a mapping $f:Q_o=[0,1]^k\to X$ into a metric space, and $x$ in the interior of $Q_o$, we define the {\em upper} and the {\em lower $n$-densities} by
$$
\Theta^{*n}(f,x)=\limsup_{r\to 0}\frac{\H^n_\infty(f(B(x,r)))}{\omega_n r^n},
\quad
\Theta^{n}_*(f,x)=\liminf_{r\to 0}\frac{\H^n_\infty(f(B(x,r)))}{\omega_n r^n},
$$
where $\H^n_\infty$ is the Hausdorff content (see Section~\ref{HHC}) and $\omega_n$ is the volume of the unit ball in $\bbbr^n$.

For a Lipschitz mapping $f:Q_o=[0,1]^{n+m}\to X$, $n\geq 1$, $m\geq 0$, into a metric space,
Azzam and Schul \cite{azzams} defined the {\em $(n,m)$-mapping content} of a set $E\subset Q_o$. However, we shall use a slightly different version of this definition that was recently introduced  by David and Schul~\cite{DS}:
$$
\H_\infty^{n,m}(f,E) =
\inf\sum_i\H^n_\infty(f(Q_i))(\diam (Q_i))^m,
$$
where the infimum is taken over all coverings $E\subset\bigcup_i Q_i\subset Q_o$ by closed dyadic cubes $Q_i$.
Since any two dyadic cubes either have disjoint interiors or one is a subset of another one, by removing unnecessary cubes, we may assume that all cubes in the covering have pairwise disjoint interiors. 

Observe that
\begin{equation}
\label{eq28}
\H^{n,m}_\infty(f,E)=\H^{n,m}_\infty(f,\tilde{E})
\quad
\text{if $\tilde{E}\subset E$ and $\H^{n+m}(E\setminus\tilde{E})=0$.}
\end{equation}
If the coverings are allowed to be by arbitrary sets, we denote the analogous content by 
$$
\hat{\H}_\infty^{n,m}(f,E) =
\inf\sum_i\H^n_\infty(f(A_i))(\diam (A_i))^m,
$$
where 
$E\subset Q_o$ , and
the infimum is taken over all coverings $E\subset\bigcup_i A_i\subset Q_o$ by arbitrary sets. Obviously, for any set $E$,
\begin{equation}
    \label{elm6}
\hat{\H}_\infty^{n,m}(f,E) \leq 
{\H}_\infty^{n,m}(f,E),
\end{equation}
however, it is not known if the two quantities are comparable \cite[Question~1.15]{DS}. 

The next result provides an equivalent definition for $\hat{\H}^{n,m}_\infty(f,E)$. For a proof see \cite[Lemma~7.3]{HajlaszE}. We will not use this result in the paper.
\begin{lemma}
If $f:Q_o=[0,1]^{n+m}\to X$ is Lipschitz and $E\subset [0,1]^{n+m}$, then
$$
\hat{\H}^{n,m}_\infty(f,E)=\inf\sum_i\frac{\omega_n}{2^n}\big(\diam f(A_i)\big)^n(\diam A_i)^m,
$$
where the infimum is taken over all coverings $E\subset\bigcup_i A_i\subset Q_o$.
\end{lemma}

In the course of the proofs we obtained other equally important results that are seemingly unrelated to the theorems listed above.
\begin{theorem}
\label{TM4}
Suppose that $f:\mathbb{R}^n\supset\Omega\to X$ is a Lipschitz continuous map from an open set onto a metric space $X$, $f(\Omega)=X$. Then, $\dim X=n$ if and only if $\mathcal{H}^n(X)>0$.
\end{theorem}
Here $\dim X$ stands for the topological dimension; see Theorem~\ref{T15}.

\begin{theorem}
\label{TM5}
Let $\Phi:X\to Y$ be a map between metric spaces that preserves the length of rectifiable curves i.e.,
$\ell_Y(\Phi\circ\gamma)=\ell_X(\gamma)$
for all rectifiable curves $\gamma:[a,b]\to X$.
Let $f:\Omega\to X$ be a locally Lipschitz map defined on an open set $\Omega\subset\bbbr^n$ for some $n$, and let $\tilde{X}=f(\Omega)$. Then
for any Borel function $g:\tilde{X}\to [0,\infty]$ we have
$$
\int_{\tilde{X}} g(x)\, d\H^n(x)=
\int_{\Phi(\tilde{X})} \Big(\sum_{x\in\Phi^{-1}(y)\cap\tilde{X}}g(x)\Big)\, d\H^n(y).
$$
\end{theorem}
See Theorem~\ref{T23}, and Theorem~\ref{T18} for a more general statement. Note that the theorem holds for any $n$. The result looks like the area formula under the assumption that the derivative of $\Phi$ is an {\em isometry}. The only problem is that under the assumptions of the theorem the derivative of $\Phi$ is not defined. 

The next result (see also Theorem~\ref{T25})
is a simple consequence of Theorem~\ref{TM5}.
It proves that the Hausdorff measure on a countably rectifiable subset of the Heisenberg group $\bbbh^n$ equals the Lebesgue measure of projections onto $\bbbr^{2n}$, taking into account the multiplicity of the projection. To our surprise, it seems that the result has not been known before.

We say that a subset $E\subset X$ of a metric space is {\em countably $k$-rectifiable}
if there is a family of Lipschitz mappings $f_i:\bbbr^k\supset E_i\to E$
defined on measurable sets $E_i\subset\bbbr^k$ such that
$$
\H^k\Big(E\setminus\bigcup_{i=1}^\infty f_i(E_i)\Big)=0.
$$
Let $\bbbh^n$ be the Heisenberg group (see Section~\ref{Heis}),
$\pi:\bbbh^n\to\bbbr^{2n}$ be the projection onto the first $2n$ coordinates, and let $\H^k_{cc}$ be the Hausdorff measure on $\bbbh^n$ with respect to the Carnot-Carath\'eodory metric.
\begin{theorem}
\label{TM6}
Assume that a set $E\subset\bbbh^n$ is countably $k$-rectifiable for some $k\leq n$. Then for any Borel function 
$g:E\to [0,\infty]$, we have
$$
\int_{E} g(x)\, d\H^k_{cc}(x)=
\int_{\pi(E)}\Big(\sum_{x\in \pi^{-1}(y)\cap E}g(x)\Big)\, d\H^k(y).
$$
\end{theorem}
There are other new results included in the paper and we follow a convention that new results are denoted as a {\em Theorem} or a {\em Proposition}, while important known results are cited as a {\em Lemma} or a {\em Corollary}.

\subsection*{The paper is structured as follows}
In Section~\ref{Prelim} we recall basic known facts about the Hausdorff measure, rectifiable curves, metric trees, and the Heisenberg groups. The Heisenberg groups are only needed to prove Theorem~\ref{TM6} and the other parts of the paper do not use Heisenberg groups at all, so the readers who are not interested in the Heisenberg groups may skip Subsection~\ref{Heis}.

In Section~\ref{KRT} we carefully state the Kirchheim-Rademacher theorem and the Kirchheim area formula. New results in this section are Propositions~\ref{Tl14} and~\ref{T28}.

In Section~\ref{AFL} we discuss mappings between metric spaces that preserve lengths of rectifiable curves and we prove Theorem~\ref{TM5} (Theorem~\ref{T23}) and Theorem~\ref{TM6} (corollary of Theorem~\ref{T25}), as well as a more general result, Theorem~\ref{T18}. 

In Section~\ref{TD} we discuss applications of metric differentiabity of Lipschitz maps to topological dimension of metric spaces and we prove Theorem~\ref{TM4} (Theorem~\ref{T15}). This result is a consequence of known facts about topological dimension and the following new result (Theorem~\ref{T16}):
\begin{theorem}
\label{TM3}
If $f:\bbbr^n\supset\Omega\to X$ is a Lipschitz map from an open set to a metric space $X$ of topological dimension $\dim X=k$, then $\rank\md(f,x)\leq k$ for almost all $x\in\Omega$.
\end{theorem}
The material of Sections~\ref{KRT}, \ref{AFL} and~\ref{TD} prepares us for the proofs of Theorems~\ref{TM1} and~\ref{TM1B}. In fact one of the implications in Theorem~\ref{TM1} is already proved in Section~\ref{TD}. 
Since a metric tree has topological dimension $1$, it follows from Theorem~\ref{TM3} that a Lipschitz map $f$ that factors through a metric tree must satisfy $\rank \md(f,x)\leq 1$ a.e. which is implication (a)$\Rightarrow$(b) in Theorem~\ref{TM1}.

In Section~\ref{FLM} we discuss a well known and  general construction of a factorization of a Lipschitz map $f:X\to Y$ defined on a quasiconvex metric space. 
The new result is Theorem~\ref{T27}. This construction is used in Section~\ref{PTM1} to prove implication (b)$\Rightarrow$(a) of Theorem~\ref{TM1}. 

Thus in Section~\ref{PTM1} we use results from all previous sections and prove the following result which is the equivalence of (a) and (b) in Theorem~\ref{TM1}.
\begin{theorem}
\label{TM7}
If $f:[0,1]^{n}\to X$, $n\geq 1$, is a Lipschitz map into a metric space, then $f$ factors through a metric tree if and only if 
$\rank\md(f,x)\leq 1$ almost everywhere.
\end{theorem}
Finally in Section~\ref{PTM3} we prove Theorem~\ref{TM1B} which along with Theorem~\ref{TM7} completes the proof of Theorem~\ref{TM1}.

\subsection*{Notation} Notation is quite standard. We denote by $\H^s$ and $\H^s_\infty$ the Hausdorff measure and the Hausdorff content respectively. We normalize $\H^s$ so that for $s=n$ it coincides with the Lebesgue measure on $\bbbr^n$, and we use $\H^n$ to denote the Lebesgue measure in $\bbbr^n$.  
The volume of the unit ball in $\bbbr^n$ is denoted by $\omega_n$.
By a null set we mean a set of measure zero.
The integral average of $f$ is denoted by the barred integral $\mvint_A f\, d\mu:=\mu^{-1}(A)\int_A f\, d\mu$.
Open and closed balls are denoted by
$B(x,r)$ and $\bar{B}(x,r)$. The unit ball and the unit sphere (centered at $0$) in $\bbbr^n$ will be denoted by $\bbbb^{n}$ and $\bbbs^{n-1}$. The characteristic function of a set $A$ is denoted by $\chi_A$. The (small inductive) topological dimension of $X$ is denoted by $\dim X$. We write $A\lesssim_{n,m}B$ if there is a constant $C>0$ that depends on $n$ and $m$ only, such that $A\leq CB$.

\subsection*{Acknowledgement} The authors would like to thank the referee for valuable comments.

\section{Preliminaries}
\label{Prelim}

\subsection{Hausdorff measure, Hausdorff content, coarea formula.}
\label{HHC}
In this section we briefly recall basic definitions and facts regarding the Hausdorff measure. For more details and proofs, see e.g., \cite{EG,federer,mattila,simon}.

Let $\omega_s=\pi^{s/2}/\Gamma(\frac{s}{2}+1)$, so $\omega_n$ is the volume of the unit ball in $\bbbr^n$ when $n$ is a positive integer. The {\em Hausdorff measure} $\H^s(E)$, $0\leq s<\infty$, of a set $E$ in a metric space is defined as $\H^s(E)=\lim_{\delta\to 0^+}\H^s_\delta(E)$, where 
\begin{equation}
\label{eq17}
\H^s_\delta(E) = \inf\sum_i\frac{\omega_s}{2^s}(\diam A_i)^s,
\qquad
0<\delta\leq\infty,
\end{equation}
and the infimum is taken over all coverings $E\subset\bigcup_i A_i$ by sets with finite diameter bounded by $\delta$, $\diam A_i\leq \delta$. Note that $\H^s_\infty$, called the {\em Hausdorff content} is defined as infimum of the sums \eqref{eq17}, where we place no restriction on the finite diameters of the sets $A_i$ covering the set $E$. The Hausdorff content is convenient when proving that a set has Hausdorff measure zero, since $\H^s(E)=0$ if and only if $\H^s_\infty(E)=0$. Also, observe that $\H^0(E)$ equals the cardinality of the set $E$.

The Hausdorff measure is an outer measure defined on all subsets of $X$ and all Borel sets are $\H^s$-measurable. $\H^s$ is {\em Borel-regular} in the following sense (see for example \cite[Lemma~2.10]{HajlaszE}).
\begin{lemma}
\label{T35}
For $s\in [0,\infty)$ and every set $E\subset X$, there is a Borel set $\tilde{E}$ such that $E\subset\tilde{E}$ and $\H^s(E)=\H^s(\tilde{E})$.
\end{lemma}
We will also need \cite[Theorem~2.6]{simon}.
\begin{lemma}
\label{T36}
$\H^n=\H^n_\infty$ on all subsets of $\bbbr^n$, and $\H^n=\H^n_\infty=\mathcal{L}^n$ on all Lebesgue measurable sets in $\bbbr^n$.
\end{lemma}

If $(X,\mu)$ is a measure space, and $f:X\to [0,\infty]$ is defined $\mu$-a.e., but  is not necessarily measurable, then its {\em upper integral} is defined as
$$
\int_X^*f\, d\mu=\inf\int_X\phi\, d\mu,
$$
where the infimum is taken over all $\mu$-measurable functions $\phi$ satisfying $0\leq f(x)\leq\phi(x)$ for $\mu$-a.e. $x\in X$.
It is important to note that if $\int_X^*f\,d\mu=0$, then $f=0$ $\mu$-a.e. (and hence $f$ is measurable).

The next result is the classical {\em coarea formula} due to Federer, see \cite{EG,federer}.
\begin{lemma} 
\label{T31}
If $f:\bbbr^{n+m}\to\bbbr^n$, $m\geq 0$, is Lipschitz and $E\subset\bbbr^{n+m}$ is measurable, then 
$$
\int_E |J_f(x)|\, d\H^{n+m}(x)=
\int_{\bbbr^n}\H^m(f^{-1}(y)\cap E)\, d\H^n(y),
\text{ where $|J_f(x)|=\sqrt{\det (Df)(Df)^T}$.}
$$
\end{lemma}

\subsection{Rectifiable curves}
A {\em curve} in a metric space $(X,d)$ is a continuous map $\gamma:[a,b]\to X$.
The {\em length} of $\gamma$ is defined as 
$$
\ell(\gamma)=\sup\left\{\sum_{i=0}^{n-1} d(\gamma(t_i),\gamma(t_{i+1}))\right\}, 
$$
where the supremum is over all $n\in\bbbn$ and all partitions $a=t_0<t_1<\ldots<t_n=b$. A curve is \textit{rectifiable} if $\ell(\gamma)<\infty$.
We will also use notation $\ell_X(\gamma)$.
For more information about rectifiable curves in metric spaces, see e.g. \cite{hajlasz1}.

Every rectifiable curve can be reparametrized as a Lipschtz curve \cite[Theorem~3.2]{hajlasz1} so without loss of generality we may assume that rectifiable curves are Lipschitz continuous.

A {\em length space} is a metric space such that the distance between any two points equals the infimum of lengths of curves connecting these two points and the space is a {\em geodesic space} if for any two points, there is a curve that connects these points and whose length equals the distance between the two points. Clearly, any geodesic space is a length space.
A shortest curve connecting given two points (if it exists) is called a {\em geodesic}. 

A metric space is {\em proper} if bounded and closed sets are compact. Proper spaces are also known as {\em boundedly compact} spaces.
The following fact is well know \cite[Theorem~3.9]{hajlasz1}.
\begin{lemma}
\label{T1}
If a metric space $X$ is proper, and $x,y\in X$ are two points which can be connected by a rectifiable curve, then there is a shortest curve connecting $x$ to $y$.
\end{lemma}
\begin{corollary}
\label{T2}
Any proper length space is geodesic. In particular compact length spaces are geodesic.
\end{corollary}
The {\em speed} of a Lipschitz curve $\gamma:[a,b]\to X$ is defined as
$$
|\dot{\gamma}|(t)=\lim_{h\to 0}\frac{d(\gamma(t+h),\gamma(t))}{|h|}\, .
$$
The next result is well known, see e.g., \cite[Theorem~3.6]{hajlasz1}.
\begin{lemma}
\label{T3}
If $\gamma:[a,b]\to X$ is Lipschitz, then the speed $|\dot{\gamma}|(t)$ exists for almost all $t\in [a,b]$ and
$$
\ell(\gamma)=\int_a^b |\dot{\gamma}|(t)\, dt.
$$
\end{lemma}
\subsection{$\bbbr$-trees}
\label{TR}
The next result provides several equivalent conditions. A metric space that satisfies any of these conditions is called a {\em metric tree} or an $\bbbr$-{\em tree}. 
\begin{lemma}
\label{T33}
Let $X$ be a geodesic space. Then the following conditions are equivalent.
\begin{itemize}
\item[(a)] For any $x,y\in X, x \neq y$, there is a unique arc with endpoints $x$ and $y$. 
\item[(b)] No subset of $X$ is homeomorphic to $\bbbs^1$.
\item[(c)] $X$ is simply connected and $\dim X=1$ (see Section~\ref{TD}).
\item[(d)] Every geodesic triangle is isometric to a tripod.
\item[(e)] $X$ is $0$-hyperbolic in the sense of Gromov.
\item[(f)] Intersection of any two closed balls is a closed ball or an empty set. 
\item[(g)] For all Lipschitz maps $\gamma:\bbbs^1\to X$ and $\pi:X\to\bbbr^2$, we have
$$
\int_{\bbbs^1}(\pi\circ\gamma)^*(x\, dy)=0.
$$
\end{itemize}
\end{lemma}
A subset of a metric space is called an {\em arc} if it is homeomorphic to the interval [0,1]. The {\em endpoints} of an arc, are the images of $0$ and $1$.
A {\em tripod} is a geodesic space consisting of three segments that meet at a common endpoint. We will not recall the definition of the Gromov hyperbolic space since we will not use it in the paper. We collected the equivalent conditions for reader's convenience and in fact we will mainly need condition (g).

For equivalence between (d), (e) and (f) and (g), see \cite{wenger}. For equivalence of (a) and (b) see \cite[Proposition 2.3]{chiswell}. Finally, the equivalence between (a) and (c) and (e) can be found in \cite{andreevb}.

\subsection{The Heisenberg groups}
\label{Heis}
Material of this section will be only used in the proof of Theorem~\ref{TM6} (Theorem~\ref{T25}), which is an application of Theorem~\ref{TM5}, and will not play any role in the other parts of the paper, so the reader may skip this section.

For any positive integer $n$,
we define the Heisenberg group  as
$\bbbh^n = \bbbr^n \times \bbbr^n\times\bbbr = \bbbr^{2n+1}$, with
the group law defined by
$$
(x,y,t)*(x',y',t')=
(x+x',t+y',t+t'+2\sum_{j=1}^n(y_jx_j'-x_jy_j')).
$$
This is a Lie group with a
basis of left invariant vector fields given 
at any point $(x_1,y_1,\dots,x_n,y_n,t) \in \bbbh^n$ by
$$
X_j=\frac{\partial}{\partial x_j}+ 2y_j\, \frac{\partial}{\partial t},
\quad
Y_j=\frac{\partial}{\partial y_j}- 2x_j\, \frac{\partial}{\partial t},
\quad
T=\frac{\partial}{\partial t}
\quad
j=1,2,\dots,n.
$$
The Heisenberg group is equipped with the left invariant Riemannian metric $g$ such that the the vector fields $X_j,Y_j,T$ are orthonormal.

The Heisenberg group is equipped with the so called 
\emph{horizontal distribution}
$$
H_p \bbbh^n = \text{span} \left\{ X_1(p),Y_1(p),\dots,X_n(p),Y_n(p) \right\} \quad \text{for all } p \in \bbbh^n.
$$
This is a smooth distribution of 
$2n$-dimensional subspaces in the 
$(2n+1)$-dimensional tangent space $T_p \bbbh^n = T_p \bbbr^{2n+1}$.
A vector $v \in T_p \bbbr^{2n+1}$ is \emph{horizontal}
if $v \in H_p \bbbh^n$.

An Lipschitz curve $\gamma:[a,b] \to \mathbb{R}^{2n+1}$ 
is a \emph{horizontal curve} if it is almost everywhere tangent to the horizontal distribution i.e.,
$\gamma'(t) \in H_{\gamma(t)} \bbbh^n$ for almost every $t \in [a,b]$. It is well known that any two points in $\bbbh^n$ can be connected by a horizontal curve.
The \emph{Carnot-Carath\'{e}odory metric} $d_{cc}$ in $\bbbh^n$
is defined as the infimum of lengths (computed with respect to the metric $g$) of horizontal curves connecting given two points.
When we talk about the Heisenberg group, we always regard it as a metric space with the Carnot-Carath\'eodory metric $d_{cc}$.
The length of a rectifiable curve $\gamma$ in $(\bbbh^n,d_{cc})$ will be denoted by $\ell_{cc}(\gamma)$. 

Let $\pi:\bbbr^{2n+1}\to\bbbr^{2n}$, $\pi(x,y,t)=(x,y)$ be the projection onto the first $2n$-coordinates. The next result is well known.
\begin{lemma}
\label{T24}
If $\gamma:[a,b]\to \bbbh^n$ is Lipschitz continuous, then $\pi\circ\gamma:[a,b]\to\bbbr^{2n}$ is Lipschitz continuous and $\ell_{cc}(\gamma)=\ell(\pi\circ\gamma)$.
\end{lemma}
In other words, the projection $\pi:\bbbh^n\to\bbbr^{2n}$ preserves lengths of Lipschitz curves.

We will need the following nontrivial Lipschitz extension result in the proof of Theorem~\ref{T25}. It is a corollary to \cite[Theorem 1.2]{wengery2}.
\begin{lemma}[Wenger-Young]
\label{elm5}
If $k \leq n$, then the the pair $(\bbbr^k, \bbbh^n)$ has the Lipschitz extension property, i.e.\ there is a constant $C>0$ such that for any $A\subset \bbbr^k$ and any $L$-Lipschitz map $f\colon A \to \bbbh^n$, there is a $CL$-Lipschitz map $F\colon \bbbr^k \to \bbbh^n$ satisfying $F(x)=f(x)$ for all $x\in A$.
\end{lemma}

\section{Kirchheim-Rademacher theorem and the area formula}
\label{KRT}
The new results in this section are  Propositions~\ref{Tl14} and~\ref{T28}. Another novelty is the emphasis on the componentwise derivative which makes geometric applications easy. The componentwise derivative has been previously investigated in \cite{hajlaszm,hajlaszMM,HZ}, but without connection to the metric derivative.

If $f:\bbbr^n\to\bbbr^m$ is differentiable at $x\in\bbbr^n$, then it follows from the triangle inequality that
$$
\lim_{y\to x}\frac{|f(y)-f(x)|-\md (f,x)(y-x)}{|y-x|}=0,
\quad
\text{where}
\quad
\md (f,x)(v)=|Df(x)v|.
$$
Note that $\md (f,x)(\cdot)$ is a seminorm. Recall that $\sigma:\bbbr^n\to [0,\infty)$ is a {\em seminorm} if $\sigma(\lambda v)=|\lambda|\sigma(v)$ and $\sigma(v+w)\leq\sigma(v)+\sigma(w)$ for all $\lambda\in\bbbr$ and $v,w\in\bbbr^n$. Unlike a norm, a seminorm may vanish on a non-trivial linear subspace of $\bbbr^n$,
$N_\sigma:=\{v\in\bbbr^n:\, \sigma(v)=0\}$, and we define
the {\em rank of a seminorm} $\sigma$ on $\bbbr^n$ as
$\rank\sigma=n-\dim N_\sigma=\dim N_\sigma^\perp$.
That is, it is the maximal dimension of a linear subspace on which $\sigma$ is a norm.

Let $f:\Omega\to X$ be a map from an open set $\Omega\subset\bbbr^n$ to a metric space $(X,d)$. We say that $f$ is {\em metrically differentiable} at $x\in\Omega$, if there is a seminorm $\md (f,x)(\cdot)$ on $\bbbr^n$ such that 
$$
\lim_{y\to x}\frac{d(f(y),f(x))-\md (f,x)(y-x)}{|y-x|}=0.
$$
If we substitute $y$ with $y=x+tv$, $v\in\bbbr^n$ and $f$ is metrically differentiable at $x$, then
\begin{equation}
\label{eq10}
\md(f,x)(v)=\lim_{t\to 0}\frac{d(f(x+tv),f(x))}{|t|}
\quad
\text{for all $v\in\bbbr^n$.}
\end{equation}
Hence, for every $v\in\bbbr^n$, the speed of the curve $t\mapsto f(x+tv)$ exists at $t=0$ and as a function of $v$ it defines a seminorm on $\bbbr^n$.

The above definition is due to Kirchheim \cite{kir}, who proved the following important generalization of the Rademacher theorem known as the {\em Kirchheim-Rademacher theorem}.
\begin{lemma}[Kirchheim]
\label{T4}
If $f:\Omega\to X$ is a Lipschitz continuous map from an open set $\Omega\subset\bbbr^n$, to a metric space $X$, then $f$ is metrically differentiable a.e.
\end{lemma}
Recall that $\ell^\infty$ is the Banach space of all bounded real sequences $x=(x_i)_{i=1}^\infty$, with the norm $\Vert x\Vert_\infty=\sup_i |x_i|$. 
\begin{lemma}[Kuratowski-Fr\'echet]
\label{T5}
Every separable metric space admits an isometric embedding into $\ell^\infty$.
\end{lemma}

Although, it is not assumed that the space $X$ in Lemma~\ref{T4} is separable, the subset $f(\Omega)\subset X$ is separable, and hence we can assume  without loss of generality that $X=\ell^\infty$. Then, Lemma~\ref{T4} is a consequence of the following more detailed result \cite{ambrosiok,kir}.
\begin{lemma}
\label{T6}
Let $f:\bbbr^n\supset\Omega\to\ell^\infty$, $f=(f_1,f_2,\ldots)$ be a Lipschitz mapping. Then $f$ is metrically differentiable a.e., and
\begin{equation}
\label{eq1}
\md (f,x)(v)=\sup_{i\in\bbbn}|\nabla f_i(x)\cdot v|
\quad
\text{for almost all $x\in\Omega$ and all $v\in\bbbr^n$.}
\end{equation}
\end{lemma}
Note that by Rademacher's theorem and the fact that countable union of null sets is a null set, at a.e.\ $x$ all $\nabla f_i(x)$ exist. If $f$ is as in Lemma~\ref{T6}, then we define the {\em componentwise derivative} of $f$ to be the $\infty \times n$ matrix
$$
Df(x) := 
\left \lceil
\begin{array}{c}
\nabla f_1(x) \\
\nabla f_2(x) \\
\vdots
\end{array}
\right \rceil.
$$
The next result is an easy exercise in linear algebra.
\begin{lemma}
\label{T7}
If $f=(f_1,f_2,\ldots):\bbbr^n\supset\Omega\to\ell^\infty$ is Lipschitz continuous, then for almost all $x\in\Omega$, the row rank of $Df(x)$ equals the column rank of $Df(x)$ and they equal $ \rank\md (f,x). $
\end{lemma}
In fact a stronger version of Lemma~\ref{T4} is proved in Kirchheim \cite{kir}:
\begin{lemma}
\label{elm7}
If $f:\bbbr^n\supset\Omega\to X$ is Lipschitz, then for almost all $x\in\Omega$, we have
\begin{equation}
\label{Elq15}
\lim_{\bbbr^n\times\bbbr^n\ni (y,z)\to (x,x)} 
\frac{d(f(y),f(z))-\md(f,x)(y-z)}{|x-y|+|x-z|}=0.
\end{equation}
\end{lemma}
\begin{remark}
Taking $y=x+tv$, $z=x+tw$, \eqref{Elq15} yields that for any $v,w\in\bbbr^n$
$$
\lim_{t\to 0} \frac{d(f(x+tv),f(x+tw))}{|t|}=\md(f,x)(v-w)
$$
which is a much stronger claim than the existence of the ``directional speed'' \eqref{eq10}.
\end{remark}
Now we use Lemma~\ref{elm7} to prove a result about covering of the image of a ball.
\begin{proposition}
\label{Tl14}
Let $f:\bbbr^n\supset\Omega\to X$ be $L$-Lipschitz. Let
$$
E_k=\{ x\in\Omega:\, \rank\md(f,x)=k \},
\quad
0\leq k\leq n.
$$
Then almost all points $x\in E_k$ have the following property:
For every integer $m\geq 1$, there exists a $r_{x,m}>0$ such that for all $0<r<r_{x,m}$, $f(B(x,r))$ can be covered by $m^k$ balls, each of radius
$3\sqrt{k}Lr/m$, in the case of $k>0$, and by one ball of radius $r/m$ in the case $k=0$. 
\end{proposition}
\begin{remark}
This result is similar to \cite[Lemma~2.7]{hajlaszm}. The approach in \cite{hajlaszm} uses componentwise differentiability instead of metric differentiability and as a result the proofs are different and more difficult.
\end{remark}
\begin{proof}
Assume first that $k>0$.
Let $\tilde{E}_k$ be the set of all points $x\in E_k$ such that \eqref{Elq15} holds.
Clearly, $|E_k\setminus\tilde{E}_k|=0$, and we will show that the property in the statement of the proposition is true for all $x\in\tilde{E}_k$. 
Fix $x\in\tilde{E}_k$, and let
\begin{equation}
\label{Eq32}
N=\{v\in\bbbr^n:\, \md(f,x)(v)=0\}
\quad
\text{so}
\quad
\dim N=n-k.
\end{equation}
By translating and rotating the coordinate system, we may assume that $x=0$ and that
$$
N^\perp=\operatorname{span}\{ e_1,\ldots,e_k\}
\quad
\text{and}
\quad
N=\operatorname{span}\{e_{k+1},\ldots,e_n\}.
$$
Fix any $r>0$.
Note that $B(x,r)=B(0,r)\subset [-r,r]^n=[-r,r]^k\times[-r,r]^{n-k}$.
Given an integer $m\geq 1$, divide the cube $[-r,r]^k$ into a grid of $m^k$ congruent cubes of edge length $2r/m$. Denote them by $\{ Q_\nu\}_{\nu=1}^{m^k}$. Then
$$
B(0,r)\subset \bigcup_{\nu=1}^{m^k} \big(Q_\nu\times [-r,r]^{n-k}\big),
\qquad
f(B(0,r))\subset\bigcup_{\nu=1}^{m^k}f(Q_\nu\times [-r,r]^{n-k}).
$$
Now it suffices to show that there is $r_{x,m}>0$ such that for all $r\in (0,r_{x,m})$ we have
\begin{equation}
\label{Eq29}
\diam(f(Q_\nu\times [-r,r]^{n-k})) \leq 3\sqrt{k}Lr/m.
\end{equation}
In the case of $k=n$, this follows from $\diam(f(Q_\nu)) \leq L\diam(Q_\nu)=2\sqrt{n}Lr/m$. In the cases $0<k<n$, \eqref{Eq29} follows from \eqref{Elq15}. Since $x=0$, \eqref{Elq15} implies that there is $r_{x,m}>0$ such that 
\begin{equation}
\label{Eq28}
|d(f(y),f(z))-\md(f,0)(y-z)|< \frac{\sqrt{k}Lr}{m}
\quad
\text{for all $y,z\in [-r,r]^n$, $0<r<r_{x,m}$.}
\end{equation}
In particular, it is true for $y,z\in Q_\nu\times [-r,r]^{n-k}$.

If $\pi:\bbbr^n\to N^\perp=\bbbr^{k}$ is the orthogonal projection
onto the orthogonal complement of \eqref{Eq32}, then by triangle inequality and the fact that the Lipschitz constant $L$ bounds $\md(f,\cdot)$, $$
\md(f,0)(v)\leq
\md(f,0)(\pi(v))+\underbrace{\md(f,0)(v-\pi(v))}_{0}
\leq L|\pi(v)|.
$$
If $y,z\in Q_\nu\times [-r,r]^{n-k}$, then $\pi(y),\pi(z)\in Q_\nu$ so $|\pi(y-z)|\leq 2\sqrt{k}r/m$, and hence
$$
\md(f,0)(y-z)\leq 2\sqrt{k}Lr/m
$$
which together with \eqref{Eq28} gives
$d(f(y),f(z))< 3\sqrt{k}Lr/m$, and \eqref{Eq29} follows.

Finally, if $k=0$, then $\md(f,x)=0$ and by definition of metric derivative there is $r_{x,m}>0$ such that 
$$
d(f(y),f(x))< \frac{|y-x|}{m}
\quad
\text{for all $y \in B(x,r), 0<r<r_{x,m}$.}
$$
But this shows that $f(B(x,r)) \subset B(f(x),r/m)$. 
\end{proof}

We define the Jacobian of a seminorm $\sigma:\bbbr^n\to [0,\infty)$ by
$$
J_n(\sigma)=\frac{\omega_n}{\H^n(\{x:\, \sigma(x)\leq 1\}}\, .
$$
Note that if $\sigma$ is not a norm (if $\sigma$ vanishes on a non-trivial linear subspace), then the set in the denominator is unbounded and it has infinite measure, so $J_n(\sigma)=0$ in that case. 


Kirchheim \cite{kir} proved the following important generalization of the classical area formula.
\begin{lemma}[Kirchheim]
\label{T8}
Let $f:\Omega\to X$ be a Lipschitz mapping from an open set $\Omega\subset\bbbr^n$ to a metric space $X$. Then
\begin{equation}
\label{eq14}
\int_\Omega g(x)J_n(\md (f,x))\, d\H^n(x)=\int_{f(\Omega)}\Big(\sum_{x\in f^{-1}(y)}g(x)\Big)\, d\H^n(y)
\end{equation}
for any Borel function $g:\Omega\to [0,\infty]$.
In particular
\begin{equation}
\label{eq15}
\int_E g(f(x)) J_n(\md (f,x))\, d\H^n(x)=\int_X g(y)\, \H^0(E\cap f^{-1}(y))\, d\H^n(y)
\end{equation}
for any Borel set $E\subset\Omega$ and any Borel function $g:X\to [0,\infty]$.
\end{lemma}
\begin{corollary}
\label{T9}
Let $f:\Omega\to X$ be a Lipschitz mapping from an open set $\Omega\subset\bbbr^n$ to a metric space $X$. Then $\H^n(f(\Omega))>0$ if and only if $\rank\md(f,x)=n$ (i.e., $\sigma$ is a norm) on a set of positive measure.
\end{corollary}
\begin{corollary}
\label{T10}
If $f=(f_1,f_2,\ldots):\bbbr^n\supset\Omega\to\ell^\infty$ is Lipschitz continuous, then 
$\H^n(f(\Omega))>0$ if and only if there is a set $E\subset\Omega$ of positive measure and indices $i_1<i_2<\ldots<i_n$ such that
$$
\det\left[\frac{\partial f_{i_k}(x)}{\partial x_\ell}\right]_{1\leq k,\ell\leq n}\neq 0
\quad 
\text{for all $x\in E$.}
$$
\end{corollary}
Corollary~\ref{T10} follows from Corollary~\ref{T9} and Lemma~\ref{T7}.
For a direct proof of Corollary~\ref{T10} that does not use Kirchheim's theorems, see \cite[Theorem~2.2]{hajlaszm}.

The next three lemmata will be used in the proofs of Proposition~\ref{T28}, Theorem~\ref{T16} and Theorem~\ref{TM7}.
\begin{lemma}
\label{T29}
Suppose that $g:\bbbr^m\to\bbbr^n$ is any mapping and that $f:\bbbr^n\to\bbbr^N$ is Lipschitz continuous. If $g$ and $f\circ g$ are differentiable at $x\in\bbbr^m$, then 
$\rank D(f\circ g)(x)\leq \rank Dg(x)$.
\end{lemma}
Indeed, if $L$ is the Lipschitz constant of $f$, then the directional derivatives of $f\circ g$ satisfy $|D_v(f\circ g)(x)|\leq L|D_vg(x)|$ and hence $\ker Dg(x)\subset\ker D(f\circ g)(x)$.

The lemma easily generalizes to the case of the metric derivative
\begin{lemma}
\label{T32}
If $X$, $Y$ are metric spaces and $g:\bbbr^n\supset\Omega\to X$, $f:X\to Y$ are Lipschitz mappings, then
$\rank\md(f\circ g,x)\leq\rank\md(g,x)$ for almost all $x\in\Omega$.
\end{lemma}
Indeed, it follows from \eqref{eq10} that $\md(f\circ g,x)\leq L\md(g,x)$, whenever $g$ and $f\circ g$ are metrically differentiable at $x$.

The next result is well known and it follows easily from the Brouwer fixed point theorem, see \cite[Lemma~7.23]{rudin}.
\begin{lemma}
\label{T30}
If $h:\bar{B}^n(0,\eps)\to\bbbr^n$ is continuous and $|h(x)-x|<\eps/2$ for all $|x|=\eps$, then $\bar{B}^{n}(0,\eps/2)\subset h(\bar{B}^n(0,\eps))$.
\end{lemma}
\begin{remark}
We will use Lemma~\ref{T30} in the proofs of Proposition~\ref{T28} and Theorem~\ref{T16}. The reader should compare the two proofs---finding similarities will help with a better understanding of the underlying ideas.
\end{remark}

The next result is of independent interest and it will be used in the proof of Theorem~\ref{TM7}.
\begin{proposition}
\label{T28}
Let $f:\Omega\to X$ be a Lipschitz mapping from an open set $\Omega\subset\bbbr^n$ to a metric space $X$, such that $\rank\md(f,y)\leq k$ for almost all $y\in\Omega$. If $g:U\to\Omega$ is a Lipschitz map from an open set $U\subset\bbbr^m$, then $\rank\md(f\circ g,x)\leq k$ for almost all $x\in U$.
\end{proposition}
\begin{remark}
This result is not obvious, because it may happen that the image of $g$ is contained in the set where $f$ is not metrically differentiable and therefore, we cannot even try to estimate $\rank\md(f\circ g,x)$ by $\rank \md(f,g(x))$, because $\md(f,g(x))$ might not exist. 
\end{remark}
\begin{proof}
For simplicity assume that $U=\bbbr^m$ and $\Omega=\bbbr^n$.
Suppose to the contrary that $\rank\md(f\circ g,\cdot)\geq k+1$ on a set of positive measure. We may assume that $X=\ell^\infty$, $f=(f_1,f_2,\ldots):\bbbr^n\to\ell^\infty$, so
the rank of the componentwise derivative satisfies 
$\rank D(f\circ g)\geq k+1$ on a set of positive measure. Therefore, we may find a set $E\subset\bbbr^m$ of positive measure, such that a $(k+1)\times (k+1)$ minor of $D(f\circ g)$ is non-zero in $E$, see Lemma~\ref{T7}. Without loss of generality we may assume that
\begin{equation}
\label{eq16}
\det\left[\frac{\partial (f_i\circ g)(x)}{\partial x_j}\right]_{1\leq i,j\leq k+1}\neq 0
\quad
\text{for all $x\in E$.}
\end{equation}
Fix $x_o\in E$ such that $g$ is differentiable at $x_o$. 
To complete the proof, it suffices to show that there is a neighborhood $G\subset\bbbr^n$ of $g(x_o)\in\bbbr^n$ such that the derivative of the mapping
$F=(f_1,\ldots,f_{k+1}):\bbbr^n\to\bbbr^{k+1}$ has rank $k+1$ on a set of positive measure in $G$, because this will imply that $\rank\md(f,\cdot)\geq k+1$, on a set of positive measure, see Lemma~\ref{T7}.

Without loss of generality we may assume that $x_o=0$ and $g(x_o)=0$.
From now on we restrict $g$ to the $(k+1)$-dimensional subspace generated by the first $(k+1)$-coordinates so we identify $g$ with  $g:=g|_{\bbbr^{k+1}}:\bbbr^{k+1}\to\bbbr^n$.
Since by \eqref{eq16}, $\rank D(F\circ g)(0)=k+1$ (because $0=x_o\in E$), it follows from 
Lemma~\ref{T29}, that $\rank Dg(0)\geq k+1$, and hence $\rank Dg(0)=k+1$, because $g$ is defined on $\bbbr^{k+1}$. 
By pre-composing $g$ with a suitable linear map and by choosing a coordinate system in $\bbbr^n$ so that $Dg(0)(\bbbr^{k+1})$ is the subspace of $\bbbr^n$ generated by the first $(k+1)$-coordinates, we may assume that $Dg(0)$ is the identity embedding of $\bbbr^{k+1}$ into $\bbbr^n$. All these assumptions are made to simplify notation. 

Thus $g:\bbbr^{k+1}\to\bbbr^n=\bbbr^{k+1}\times\bbbr^{n-k-1}$, and $g(x)=(x,0)+o(|x|)\in\bbbr^{k+1}\times\bbbr^{n-k-1}$. Since by \eqref{eq16}, $\det D(F\circ g)(0)\neq 0$, we may assume (after post-composing with an affine isomorphism) that $F(0)=0$ and $D(F\circ g)(0)=\id$ i.e., $(F\circ g)(x)=x+o(|x|)$.

Therefore, there is $\eps>0$, such that if $|x|=\eps$, then 
$$
|(F\circ g)(x)-x|<\frac{\eps}{6} 
\quad
\text{and}
\quad
|g(x)-(x,0)|<\frac{\eps}{6L},
$$
where $L$ is a Lipschitz constant of $F$.

Fix $y\in \bbbr^{n-k-1}$, $|y|<\frac{\eps}{6L}$, such that
$F$ is differentiable at almost all points of the hyperplane $\bbbr^{k+1}\times\{ y\}\subset\bbbr^n$ (by Fubini's theorem almost all $y\in\bbbr^{n-k-1}$ have this property).
For $|x|=\eps$ we have
\begin{equation*}
\begin{split}
|F(x,y)-x|
&\leq 
|F(x,y)-F(x,0)|+|F(x,0)-F(g(x))|+|(F\circ g)(x)-x|\\
&<
L|y| + L|(x,0)-g(x)|+\frac{\eps}{6}<\frac{\eps}{2}.
\end{split}
\end{equation*}
It follows from Lemma~\ref{T30} that
$$
\bar{B}^{k+1}(0,\eps/2)\subset F(\bar{B}^{k+1}(0,\eps)\times \{ y\}).
$$
In particular, the $(k+1)$-dimensional measure of $F(\bar{B}^{k+1}(0,\eps)\times \{ y\})$ is positive and it follows from the classical area formula that 
$$
\rank D\Big(F\big|_{\bar{B}^{k+1}(0,\eps)\times\{ y\}}\Big)= k+1
$$
on a set of positive measure. Since it is true for almost all $y\in\bbbr^{n-k-1}$ such that $|y|<\frac{\eps}{6L}$, it follows that $\rank DF\geq k+1$ on a subset of
$B^{k+1}(0,\eps)\times B^{n-k-1}(0,\eps/6L)$ of
positive measure and we arrive at a contradiction.
\end{proof}

\section{Area formula for length preserving mappings}
\label{AFL}

Let us start with a simplified and a more transparent version of the main result of this section which is Theorem~\ref{T18}.
\begin{theorem}
\label{T23}
Let $\Phi:X\to Y$ be a map between metric spaces that preserves the length of rectifiable curves i.e.,
$\ell_Y(\Phi\circ\gamma)=\ell_X(\gamma)$
for all rectifiable curves $\gamma:[a,b]\to X$.
Let $f:\Omega\to X$ be a locally Lipschitz map defined on an open set $\Omega\subset\bbbr^n$ for some $n$, and let $\tilde{X}=f(\Omega)$. Then
for any Borel function $g:\tilde{X}\to [0,\infty]$ we have
$$
\int_{\tilde{X}} g(x)\, d\H^n(x)=
\int_{\Phi(\tilde{X})} \Big(\sum_{x\in\Phi^{-1}(y)\cap\tilde{X}}g(x)\Big)\, d\H^n(y).
$$
\end{theorem}
\begin{remark}
We do not assume that $\Phi$ is continuous (however, $\Phi\circ f$ is locally Lipschitz continuous). For example, if the only rectifiable curves in $X$ are constant ones (e.g., if $X$ is the von Koch snowflake or a Cantor set), then {\em any} map $\Phi:X\to Y$ satisfies the assumptions of the theorem. However, the result is trivial in that case since $\tilde{X}=f(\Omega)$ consists of a single point (if $\Omega$ is connected). Thus the result is interesting only if there are many rectifiable curves in $X$.
\end{remark}
\begin{remark}
Note that the theorem holds for any $n$. The result looks like the area formula under the assumption that the derivative of $\Phi$ is an {\em isometry}. The only problem is that under the assumptions of the theorem the derivative of $\Phi$ is not and cannot be defined. 
\end{remark}

Since according to Lemma~\ref{T24}, the projection $\pi:\bbbh^n\to\bbbr^{2n}$ preserves length of rectifiable curves, we obtain \begin{theorem}
\label{T25}
Let $f:K\to\bbbh^n$ be a Lipschitz map defined on a Borel set $K\subset\bbbr^k$ for some $k\leq n$. Then for any Borel function 
$g:f(K)\to [0,\infty]$ we have
$$
\int_{f(K)} g(x)\, d\H^k_{cc}(x)=
\int_{\pi(f(K))}\Big(\sum_{x\in \pi^{-1}(y)\cap f(K)}g(x)\Big)\, d\H^k(y).
$$
\end{theorem}
Indeed, this follows immediately from Theorem~\ref{T23} because according to Lemma~\ref{elm5}, we may assume that $f$ is defined on $\bbbr^k$. Now, Theorem~\ref{TM6} follows immediately from Theorem~\ref{T25}, because any countably $k$-rectifiable subset of $\bbbh^n$ is the union of countably many disjoint sets $f(K_i)$ plus a set of $\H^k_{cc}$-measure zero.

Theorem~\ref{T23} is a straightforward consequence of the following more general result. While the statement of Theorem~\ref{T18} is not as appealing as that of Theorem~\ref{T23}, we actually need this more general statement for the applications to results in Section~\ref{FLM}, see Theorem~\ref{T27}. 

\begin{theorem}
\label{T18}
Let $\Phi:X\to Y$ be a map between metric spaces. Assume that $\Omega\subset\bbbr^n$ is open and $f:\Omega\to X$, is Lipschitz. 
For $x\in\Omega$ and $v\in\bbbs^{n-1}$ let
$$
\gamma_{x,v}(t)=f(x+tv):[0,d(x)]\to X,
\quad
\text{where} 
\quad
d(x)=\min\Big\{\frac{1}{2}\dist(x,\partial\Omega),1\Big\}
$$
be a family of Lipschitz curves in $X$.
Assume that
\begin{equation}
\label{eq9}
\ell_Y(\Phi\circ\gamma_{x,v}|_{[0,t]})=\ell_X(\gamma_{x,v}|_{[0,t]})
\quad
\text{for all $x\in\Omega$, $v\in\bbbs^{n-1}$, and all $t\in [0,d(x)]$.}
\end{equation}
Then
\begin{itemize}
\item[(a)]
$\md (\Phi\circ f,x)=\md (f,x)$
for almost all $x\in\Omega$.\\
\item[(b)]
$J_n(\md(\Phi\circ f,x))=J_n(\md (f,x))$
for almost all $x\in\Omega$.\\
\end{itemize}
Moreover, if $\tilde{X}=f(\Omega)$, then
\begin{itemize}
\item[(c)]
For any Borel function $g:\tilde{X}\to [0,\infty]$.
$$
\int_{\tilde{X}} g(x)\, d\H^n(x)=
\int_{\Phi(\tilde{X})} \Big(\sum_{x\in\Phi^{-1}(y)\cap\tilde{X}}g(x)\Big)\, d\H^n(y).
$$
\item[(d)]
For any Borel set $E\subset\tilde{X}$ and any Borel function $g:Y\to [0,\infty]$
$$
\int_E (g\circ\Phi)(x)\, d\H^n(x)
=\int_Y g(y)\H^0(\Phi^{-1}(y)\cap E)\, d\H^n(y).
$$
\end{itemize}
\end{theorem}
\begin{remark}
Even if $f$ is one-to-one and surjective, $\Phi$ need not be continuous for the claim of the theorem to be true. For example, we can have $f$ defined on $(0,1)$ that bends the interval in a length preserving way, and glues $1$ to $1/2$. Since $1$ is not a point in the domain $(0,1)$ the map is one-to-one and the inverse map $\Phi=f^{-1}$ is discontinuous at $1/2$, but it preserves the length of curves $\gamma_{x,v}$. 
The lack of continuity does not create any problem in the proof, because $\Phi\circ f$ is locally Lipschitz continuous. 
\end{remark}

In the proof we will need three lemmata.

From \eqref{eq10} we know that $\md(f,x)(v)$ equals the speed of the curve $\gamma_{x,v}(t)$ at $t=0$. However, assumptions of the theorem provide information about length of curves so it will be convenient to express $\md(f,x)(v)$ as a derivative of the length of the curve $\gamma_{x,v}(t)$. The lemmata will be focused on that problem.

Given a Lipschitz curve $\gamma:[a,b]\to X$,
let $s_\gamma:[a,b]\to [0,\ell(\gamma)]$, $s_\gamma(t)=\ell(\gamma|_{[a,t]})$, be the so-called {\em arc-length} parameter. 
\begin{lemma}
\label{T19}
If $\gamma:[a,b]\to X$ is Lipschitz, then $s_\gamma$ is Lipschitz and $\dot{s}_\gamma(t)=|\dot{\gamma}|(t)$ for almost all $t\in [a,b)$. That is,
\begin{equation}
\label{eq11}
|\dot{\gamma}|(t)=\lim_{h\to 0^+}\frac{d(\gamma(t+h),\gamma(t))}{h}= \lim_{h\to 0^+}\frac{\ell(\gamma|_{[t,t+h]})}{h}
\quad
\text{for almost all $t\in [a,b)$.}
\end{equation}
\end{lemma}
\begin{proof}
Let $\gamma$ be $L$-Lipschitz.
For $a\leq t_1\leq t_2\leq b$ we have
$$
|s_\gamma(t_2)-s_\gamma(t_1)|=\ell(\gamma|_{[t_1,t_2]})\leq L|t_2-t_1|
$$
which proves that $s_\gamma$ is Lipschitz. In particular $s_\gamma$ is differentiable almost everywhere. Since the length of a curve connecting two points is no less than the distance between the points, $\dot{s}_\gamma\geq |\dot{\gamma}|$ almost everywhere. This and the equality
$$
\int_a^b |\dot{\gamma}|(t)\, dt = \ell(\gamma)=s_\gamma(b)-s_\gamma(a)=
\int_a^b \dot{s}_\gamma(t)\, dt
$$
proves that $\dot{s}_\gamma=|\dot{\gamma}|$ almost everywhere which is \eqref{eq11}.
\end{proof}
\begin{lemma}
\label{T20}
Let $\{v_i\}_{i=1}^\infty\subset\bbbs^{n-1}$ be a countable and a dense subset of the (Euclidean) unit sphere.
Let $f:\bbbr^n\supset\Omega\to X$ be a Lipschitz mapping from an open set into  
a metric space. For $x\in\Omega$ and $v\in \bbbs^{n-1}$ let $\gamma_{x,v}(t)=f(x+tv)$. Then for almost all $x\in\Omega$ and all $i=1,2,\ldots$ we have
$$
\md (f,x)(v_i)=\lim_{t\to 0^+}\frac{\ell(\gamma_{x,v_i}|_{[0,t]})}{t}\, .
$$
\end{lemma}
\begin{proof}
For simplicity of notation assume that $\Omega=\bbbr^n$. 
Fix $v\in \bbbs^{n-1}$. It suffices to prove that
\begin{equation}
\label{eq33}
\md(f,x)(v)=
\lim_{t\to 0^+}\frac{\ell(\gamma_{x,v}|_{[0,t]})}{t}
\quad
\text{for almost all $x\in\bbbr^n$,}
\end{equation}
because the result will be a straightforward consequence of the fact that the union of countably many sets of measure zero has measure zero.

Let
$
W=v^\perp=\{w\in\bbbr^n:\, \langle w,v\rangle=0\}.
$
Let $\Gamma:W\times\bbbr\to X$ be defined by 
$$
\Gamma(w,t)=\Gamma_w(t)=f(w+tv).
$$
According to Lemma~\ref{T19}, for every $w\in W$ and almost all $t\in\bbbr$,
\begin{equation}
\label{eq12}
|\dot{\Gamma}_w|(t)=\lim_{h\to 0^+}\frac{\ell(\Gamma_w|_{[t,t+h]})}{h}.
\end{equation}
The Fubini theorem implies that the set $\tilde{E}\subset W\times\bbbr$ of points $(w,t)$ for which \eqref{eq12} does not hold has measure zero.
The mapping $\Phi:W\times\bbbr\to\bbbr^n$, $\Phi(w,t)=w+tv$ is a linear isometry and hence $E=\Phi(\tilde{E})$ has measure zero.

If $x\in\bbbr^n\setminus E$, $x=\Phi(w,t)=w+tv$, then
$\Gamma_w(t+h)=\gamma_{x,v}(h)$ so  \eqref{eq12} yields
$$
|\dot{\gamma}_{x,v}|(0)=|\dot{\Gamma}_w|(t)=
\lim_{h\to 0^+}\frac{\ell(\Gamma_w|_{[t,t+h]})}{h}=
\lim_{h\to 0^+}\frac{\ell(\gamma_{x,v}|_{[0,h]})}{h},
$$
and \eqref{eq33} follows, because according to \eqref{eq10}, for almost all $x\in\bbbr^n$ we have
$$
\md(f,x)(v)=
\lim_{t\to 0}\frac{d(f(x+tv),f(x))}{t}=|\dot{\gamma}_{x,v}|(0).
$$
The proof is complete.
\end{proof}
Lemma~\ref{T20} describes values of the seminorm $\md (f,x)$ on a countable and dense subset of $\bbbs^{n-1}$ and the next lemma shows that this information
completely determines a seminorm.
\begin{lemma}
\label{T21}
Let $\{v_i\}_{i=1}^\infty\subset\bbbs^{n-1}$ be a countable and a dense subset.
If $\sigma_1$, $\sigma_2$ are seminorms on $\mathbb{R}^n$, and $\sigma_1(v_i)=\sigma_2(v_i)$ for all $i=1,2,\ldots$, then $\sigma_1(v)=\sigma_2(v)$ for all $v\in\mathbb{R}^n$.
\end{lemma}
\begin{proof}
Since $\sigma_1(tv_i)=\sigma_2(tv_i)$ for all $i\in\bbbn$ and $t\in\bbbr$, it follows that the equality holds on the set $E=\{tv_i:\, t\in\bbbr,i\in\bbbn\}$
that is dense in $\bbbr^n$ in the Euclidan metric.
It is a routine exercise to show that any seminorm $\sigma$ is bounded by the Euclidean norm, that is $\sigma(v)\leq C|v|$. Therefore,
$|\sigma(u)-\sigma(v)|\leq \sigma(u-v)\leq C|u-v|$. For $v\in\bbbr^n$ choose $w_k\in E$ such that $|v-w_k|\to 0$. Then $\sigma(w_k)\to \sigma(v)$, so passing to the limit in  
$\sigma_1(w_k)=\sigma_2(w_k)$ as $k\to\infty$, yields the result.
\end{proof}

\begin{proof}[Proof of Theorem~\ref{T18}]
Fix a countable and a dense set $\{ v_i\}_{i=1}^\infty\subset\bbbs^{n-1}$. According to Lemma~\ref{T20}, for almost all $x$ and all $i$,
$$
\md (f,x)(v_i)=\lim_{t\to 0^+}\frac{\ell_X(\gamma_{x,v_i}|_{[0,t]})}{t}
\quad
\text{and}
\quad
\md (\Phi\circ f,x)(v_i)=\lim_{t\to 0^+}\frac{\ell_Y(\Phi\circ\gamma_{x,v_i}|_{[0,t]})}{t}\, .
$$
Therefore \eqref{eq9} implies that
$$
\md(\Phi\circ f,x)(v_i)=\md (f,x)(v_i)
$$
and (a) follows from Lemma~\ref{T21}, while (b) is an immediate consequence of (a).
Finally, we show that (c) and (d) are consequences of Lemma~\ref{T8}.
Since (d) easily follows from (c) it remains to prove (c).

Given a Borel function $g:\tilde{X}\to [0,\infty]$, let $G:\tilde{X}\to [0,\infty]$ be defined by
$$
G(y)=\frac{g(y)}{\H^0(f^{-1}(y))}\, .
$$
Then, we have
\begin{equation*}
\begin{split}
& 
\int_{\tilde{X}} g(y)\, d\H^n(y) =
\int_{\tilde{X}} G(y)\H^0(f^{-1}(y))\, d\H^n(y) \stackrel{\eqref{eq15}}{=}
\int_\Omega G(f(x))J_n(\md(f,x))\, d\H^n(x)\\
&=
\int_\Omega G(f(x))J_n(\md(\Phi\circ f,x))\, d\H^n(x) \stackrel{\eqref{eq14}}{=}
\int_{\Phi(\tilde{X})} \sum_{x\in(\Phi\circ f)^{-1}(y)} G(f(x))\, d\H^n(y).
\end{split}    
\end{equation*}
We used here the fact that $\Phi\circ f$ is locally Lipschitz continuous.
It remains to observe that
$$
(\Phi\circ f)^{-1}(y)=f^{-1}(\Phi^{-1}(y))=\bigcup_{z\in \Phi^{-1}(y)} f^{-1}(z)=\bigcup_{z\in \Phi^{-1}(y)\cap\tilde{X}} f^{-1}(z),
$$
because $f^{-1}(z)=\varnothing$ if $z\not\in\tilde{X}$,
and hence for $y\in \Phi(\tilde{X})$ we have
\begin{equation*}
\begin{split}
\sum_{x\in(\Phi\circ f)^{-1}(y)} G(f(x)) 
& =
\sum_{z\in\Phi^{-1}(y)\cap\tilde{X}}\sum_{x\in f^{-1}(z)} G(f(x))\\
&=
\sum_{z\in\Phi^{-1}(y)\cap\tilde{X}} G(z)\H^0(f^{-1}(z))
=
\sum_{z\in\Phi^{-1}(y)\cap\tilde{X}} g(z).
\end{split}
\end{equation*}
The proof is complete.
\end{proof}

\section{Topological dimension}
\label{TD}
The {\em topological dimension} (small inductive dimension) $\dim X$ of a metric space $X$ is defined as follows:
\begin{itemize}
\item $\dim X$ is an integer greater than or equal to $-1$ or $\dim X=\infty$.
\item $\dim X=-1$ if and only if $X=\varnothing$.
\item $\dim X\leq n$ if every point in $X$ has an arbitrarily small neighborhood whose boundary has dimension $\leq n-1$. 
\item $\dim X=n$ if $\dim X\leq n$ and it is not true that $\dim X\leq n-1$.
\item $\dim X=\infty$ if $\dim X\leq n$ is false for all integers $n\geq -1$.
\end{itemize}
There are many other definitions of the topological dimension. They are equivalent to the above one if $X$ is a separable metric space, see \cite{engelking,HW}.

The next result is well known; see e.g., \cite[Theorem~2]{andreevb}.
\begin{lemma}
\label{T11}
If an $\bbbr$-tree $T$ has at least two points, then $\dim T=1$.
\end{lemma}
\begin{proof}[Sketch of a proof]
Clearly $\dim T\geq 1$, since $T$ contains a segment. To prove that $\dim T\leq 1$ it suffices to show that boundaries of balls in $T$ have dimension $\leq 0$. It is not difficult to prove that boundaries of balls in $T$ are ultrametric spaces and every non-empty ultrametric space has dimension $0$, because balls in ultrmetric spaces are clopen.
\end{proof}

The following result is due to Szpilrajn \cite{szpil}. For a proof see \cite[Theorem~8.15]{heinonen} 
\begin{lemma}
\label{T12}
If a metric space $X$ satisfies $\H^{n+1}(X)=0$, where $n\geq -1$ is an integer, then
$\dim X\leq n$.
\end{lemma}
On the other hand any Cantor set has topological dimension $0$, but one can construct Cantor sets of infinite Hausdorff dimension. Thus in general, information about the topological dimension does not give any upper estimate for the Hausforff dimension, except the situation described in Theorem~\ref{T15}.

\begin{theorem}
\label{T15}
Suppose that $f:\mathbb{R}^n\supset\Omega\to X$ is a Lipschitz continuous map from an open set onto a metric space $X$, $f(\Omega)=X$. Then, $\dim X=n$ if and only if $\mathcal{H}^n(X)>0$.
\end{theorem}
\begin{proof}
It follows from Lemma~\ref{T12} that if $\dim X=n$, then $\H^n(X)>0$. Thus it remains to show the opposite implication. Suppose that $\H^n(X)>0$.
Since $\H^{n+1}(X)=0$ ($X$ is a Lipschitz image of a subset of $\bbbr^n$), Lemma~\ref{T12} implies that $\dim X\leq n$ and it suffices to show that the inequality $\dim X\leq n-1$ is false.
Suppose to the contrary that $\dim X\leq n-1$.
The inequality $\H^n(X)>0$ and Corollary~\ref{T9} imply that $\rank\md(f,x)=n$ on a set of positive measure. Hence the result is a direct consequence of Theorem~\ref{T16} applied to $\dim X=k\leq n-1$.
Indeed, Theorem~\ref{T16} implies that $\rank\md(f,x)\leq k\leq n-1$ a.e. which contradicts the fact that $\rank\md(f,x)=n$ on a set of positive measure.
\end{proof}

\begin{theorem}
\label{T16}
If $f:\bbbr^n\supset\Omega\to X$ is a Lipschitz map from an open set to a metric space $X$ of topological dimension $\dim X=k$, then $\rank\md(f,x)\leq k$ for almost all $x\in\Omega$.
\end{theorem}
\begin{proof}
The result is obvious if $k\geq n$ so we may assume that $k<n$. 
Let $\tilde{X}=f(\Omega)$. The space $\tilde{X}$ is separable and $\dim\tilde{X}\leq k$. Separability of $\tilde{X}$ allows us to assume that $\tilde{X}\subset\ell^\infty$ and $f=(f_1,f_2,\ldots):\Omega\to\ell^\infty$.

Suppose to the contrary that $\rank\md(f,x)\geq k+1$ on a set of positive measure. 
We will arrive to a contradiction by showing that $\dim \tilde{X}>k$.
To this end we shall need the following classical result \cite[Theorem~1.9.3]{engelking}, \cite[Theorem~ VI.4]{HW}.
\begin{lemma}
\label{T14}
A separable metric space $X$ has topological dimension less than or equal $k$, $k\geq 0$, if and only if for each closed set $C\subset X$ and a continuous map $h:C\to\bbbs^{k}$, there is a continuous extension $H:X\to\bbbs^{k}$ of $h$.
\end{lemma}
Thus it remains to find a closed set $C\subset \tilde{X}$ and a continuous map $h:C\to\bbbs^{k}$ that has no continuous extension $H:\tilde{X}\to\bbbs^{k}$.

It follows from Lemma~\ref{T7} that a certain $(k+1)\times (k+1)$ minor of the componentwise derivative $Df$ is non-zero on a set of positive measure. After relabeling indices, translating $\Omega$, and translating the image in $\ell^\infty$ we may assume that $0\in\Omega$, $f(0)=0$ and that the function
$$
F=(f_1,\ldots,f_{k+1})\big|_{\tilde{\Omega}}:\tilde{\Omega}\to\bbbr^{k+1},
\qquad
\tilde{\Omega}=\Omega\cap\{(x_1,\ldots,x_{k+1},0,\ldots,0)\}\subset\bbbr^{k+1}
$$
is differentiable at $0\in\tilde{\Omega}$ with 
$\det DF(0)\neq 0$. Further, replacing $f$ by $f\circ (DF(0))^{-1}$, we may assume that $0\in\tilde{\Omega}$, and
$$
F(0)=0,
\quad
DF(0)=I
\quad
\text{(identity matrix).}
$$
Therefore, there is $r>0$ such that $\bar{B}^{k+1}(0,r)\subset\tilde{\Omega}$ and
\begin{equation}
\label{eq2}
|F(x)-x|< r/4
\quad
\text{whenever $|x|=r$.}
\end{equation}
Let $\pi:\ell^\infty\to\bbbr^{k+1}$, $\pi(x_1,x_2,\ldots)=(x_1,x_2,\ldots,x_{k+1})$ be the projection on the first $k+1$ components, so $F=\pi\circ (f|_{\tilde{\Omega}})$. 

Let $C=f(S^{k}(0,r))\subset\tilde{X}\subset\ell^\infty$, where $S^k(0,r)=\partial\bar{B}^{k+1}(0,r)\subset\tilde{\Omega}$.
Note that if $y\in C$, then $\pi(y)\neq 0$. Indeed, $y=f(x)$, $|x|=r$ so $\pi(y)=F(x)$ and hence
$$
|\pi(y)|\geq |x|-|\pi(y)-x|=r-|F(x)-x|> \frac{3r}{4}>0.
$$
Therefore,
$$
h:C\to\bbbs^{k},
\qquad
h(y)=\frac{\pi(y)}{|\pi(y)|}
$$
is well defined and continuous. It remains to show that there is no continuous extension $H:\tilde{X}\to\bbbs^{k}$ of $h$. Suppose, by way of contradiction, that such $H$ exists. Then
\begin{equation}
\label{eq3}    
g:\bar{B}^{k+1}(0,1)\to\bbbs^{k},
\qquad
g(x)=H(f(rx))
\end{equation}
is well defined and continuous. 

If $|x|=1$, then $rx\in S^{k}(0,r)$, so $f(rx)\in C$ and hence
$$
g(x)=h(f(rx))=\frac{F(rx)}{|F(rx)|},
\quad
\text{whenever $|x|=1$.}
$$
It suffices now to show that the map
\begin{equation}
\label{eq4}
g\big|_{\bbbs^{k}(0,1)}:\bbbs^{k}\to\bbbs^{k},
\qquad
g(x)=\frac{F(rx)}{|F(rx)|}
\end{equation}
satisfies
\begin{equation}
\label{eq5}
|g(x)-x|< \frac{1}{2}
\quad
\text{whenever $|x|=1$.}
\end{equation}
Indeed, since $g:\bar{B}^{k+1}(0,1)\to\bbbs^k\subset\bbbr^{k+1}$, Lemma~\ref{T30} yields that
$\bar{B}^{k+1}(0,1/2)\subset g(\bar{B}^{k+1}(0,1))$ which contradicts the fact that the image of $g$ is contained in the unit sphere.

To prove \eqref{eq5}, let $|x|=1$. It follows from \eqref{eq2} and the triangle inequality that 
$$
\left|1-\frac{|F(rx)|}{r}\right|=
\left||x|-\Big|\frac{F(rx)}{r}\Big|\right|\leq
\left|\frac{F(rx)}{r}-x\right|<\frac{1}{4}.
$$
Therefore,
$$
|g(x)-x|\leq
\left|\frac{F(rx)}{|F(rx)|}-\frac{F(rx)}{r}\right|+
\left|\frac{F(rx)}{r}-x\right|<
\left|\frac{F(rx)}{|F(rx)|}\Big(1-\frac{|F(rx)|}{r}\Big)\right|+\frac{1}{4}< \frac{1}{2}.
$$
This proves \eqref{eq5} and completes the proof of the theorem.
\end{proof}
As a corollary of Theorem~\ref{T16}, Lemma~\ref{T11} and Corollary~\ref{T9} we obtain
\begin{theorem}
\label{T17}
If $f:\bbbr^n\supset\Omega\to T$, is a Lipschitz map from an open set into an $\bbbr$-tree $T$, then $\rank\md(f,x)\leq 1$ a.e. If in addition, $n\geq 2$, then $\H^n(f(\Omega))=0$.
\end{theorem}

\section{Factoring Lipschitz maps}
\label{FLM}
Material of this section is based mostly on \cite{wengery}. Similar 
constructions appear also in \cite{lytchakw,petrunins}.
Theorem~\ref{T27} is new.

Given metric spaces $X,Y,Z$, we say that a Lipschitz map $\Phi:X\to Y$ {\em factors through} $Z$, if there are Lipschitz maps $\psi:X\to Z$ and $\phi:Z\to Y$ such that $\Phi=\phi\circ\psi$.

We say that a metric space $(X,d)$ is {\em $C_q$-quasiconvex}, where $C_q\geq 1$ is a constant, if for any $x,y\in X$ there is a rectifiable curve $\gamma:[0,1]\to X$ such that $\gamma(0)=x$, $\gamma(1)=y$ and $\ell(\gamma)\leq C_qd(x,y)$. A meric space is said to be {\em quasiconvex} if it is $C_q$-quasiconvex for some $C_q\geq 1$.

Let $X$ be a $C_q$-quasiconvex metric space, $Y$ another metric space, and let $\Phi:X\to Y$ be an $L$-Lipschitz map. We define a quasimetric on $X$ by
\begin{equation}
\label{FTEq1}
d_\Phi(x,y)=\inf\{ \ell(\Phi\circ \gamma):\, \gamma:[0,1]\to X,\ \gamma(0)=x,\ \gamma(1)=y\},
\end{equation}
where the infimum is over all rectifiable curves $\gamma$ connecting $x$ to $y$ in $X$. Note that, with a suitable reparameterization, we can assume that $\gamma$ is Lipschitz. 

It is easy to see that
\begin{equation}
\label{FTEq2}
d_Y(\Phi(x),\Phi(y)) \leq d_\Phi(x,y)\leq C_qLd_X(x,y).
\end{equation}
In particular,
\begin{equation}
\label{FTEq3}
d_\Phi(x,y)=0
\quad
\Rightarrow
\quad
\Phi(x)=\Phi(y).
\end{equation}
However, in general,
$$
\Phi(x)=\Phi(y)
\quad
\nRightarrow
\quad
d_\Phi(x,y)=0.
$$
Let $\sim$ be an equivalence relation in $X$ defined by
$$
x\sim y
\quad
\text{if and only if}
\quad
d_\Phi(x,y)=0,
$$
and let $Z_\Phi=X/\sim$. We equip $Z$ with the quotient distance
$$
d_\Phi([x],[y]):=d_\Phi(x,y).
$$
(One needs to check first that $d_\Phi$ is well defined in $Z_\Phi$ i.e., if $x\sim x'$ and $y\sim y'$, then $d_\Phi(x,y)=d_\Phi(x',y')$.)
The next result is an easy exercise left to the reader.
\begin{lemma}
\label{FTT1}
$(Z_\Phi,d_\Phi)$ is a metric space.
\end{lemma}

Define now the mappings
\begin{align*}
&X  \stackrel{\psi}{\longrightarrow}  Z_\Phi  \stackrel{\phi}{\longrightarrow}  Y\\
& x\stackrel{\psi}{\longmapsto} [x] \stackrel{\phi}{\longmapsto} \Phi(x).
\end{align*}
Hence $\Phi=\phi\circ\psi$. Note that the mapping
$$
\phi:Z_\Phi\to Y,
\quad
\phi([x])=\Phi(x)
$$
is well defined, because \eqref{FTEq3} yields
$$
[x]=[x']
\quad
\equiv
\quad
x\sim x'
\quad
\equiv
\quad
d_\Phi(x,x')=0
\quad
\Rightarrow
\Phi(x)=\Phi(x').
$$
\begin{lemma}
\label{FTT2}
The mapping $\psi:X\to Z_\Phi$ is $C_qL$-Lipschitz and the mapping $\phi:Z_\Phi\to Y$ is $1$-Lipschitz. Therefore $\Phi:X\to Y$ factors through $Z_\Phi$, namely $\Phi=\phi\circ\psi$.
\end{lemma}
\begin{proof}
The mapping $\psi$ is $C_qL$-Lipschitz because according to \eqref{FTEq2}
$$
d_\Phi(\psi(x),\psi(y))=d_\Phi([x],[y])=d_\Phi(x,y)\leq C_qLd_X(x,y).
$$
On the other hand, $\phi$ is $1$-Lipschitz because according to \eqref{FTEq2}
$$
d_Y(\phi([x]),\phi([y]))=d_Y(\Phi(x),\Phi(y))\leq d_\Phi(x,y)=d_\Phi([x],[y]).
$$
\end{proof}	
\begin{corollary}
\label{FTT3}
For any curve $\alpha:[0,1]\to Z_\Phi$ we have $\ell(\phi\circ\alpha)\leq\ell(\alpha)$.
\end{corollary}
Indeed, $\phi$ is $1$-Lipschitz and composing with a $1$-Lipschitz map cannot increase the length of a curve.
\begin{lemma}
\label{FTT4}
If $\gamma:[0,1]\to X$ is a rectifiable curve and $\alpha=\psi\circ\gamma:[0,1]\to Z_\Phi$, then
$\ell(\alpha)=\ell(\phi\circ\alpha)$.
\end{lemma}
In other words, $\phi$ preserves lengths of curves in $Z_\Phi$ that are images of rectifiable curves in $X$. 
\begin{proof}
Let $\gamma:[0,1]\to X$ and $\alpha=\psi\circ\gamma:[0,1]\to Z_\Phi$. In view of Corollary~\ref{FTT3} it suffices to show that 
\begin{equation}
\label{FTEq6}
\ell(\phi\circ\alpha)\geq\ell(\alpha).
\end{equation}
Note that
\begin{equation}
\label{FTEq7}
\phi\circ\alpha=\Phi\circ\gamma.
\end{equation}
Indeed, $\phi\circ\alpha=\phi\circ\psi\circ\gamma=\Phi\circ\gamma$.
Now taking the supremum over the partitions $0=t_0<t_1<\ldots<t_n=1$ we get
$$
\ell(\alpha)=
\sup\sum_{i=0}^{n-1}d_\Phi([\gamma(t_i)],[\gamma(t_{i+1})])=
\sup\sum_{i=0}^{n-1} d_\Phi(\gamma(t_i),\gamma(t_{i+1}))=: \heartsuit.
$$
Since $\gamma|_{[t_i,t_{i+1}]}$ is a rectifiable curve connecting $\gamma(t_i)$ to $\gamma(t_{i+1})$, the definition of $d_\Phi$ (see \eqref{FTEq1}) yields that 
$$
d_\Phi(\gamma(t_i),\gamma(t_{i+1}))\leq \ell\left((\Phi\circ\gamma)|_{[t_i,t_{i+1}]}\right)
$$
and hence 
$$
\heartsuit\leq \sup\sum_{i=0}^{n-1} \ell\left((\Phi\circ\gamma)|_{[t_i,t_{i+1}]}\right)=\ell(\Phi\circ\gamma)\stackrel{\eqref{FTEq7}}{=}\ell(\phi\circ\alpha),
$$
This completes the proof of \eqref{FTEq6} and hence that of the lemma.
\end{proof}	

\begin{corollary}
\label{FTT5}
$(Z_\Phi,d_\Phi)$ is a length space. If in addition, $X$ is compact, then $Z_\Phi$ is a geodesic space.
\end{corollary}
\begin{proof}
If $[x],[y]\in Z_\Phi$ and $\gamma:[0,1]\to X$, $\gamma(0)=x$, $\gamma(1)=y$, is a rectifiable curve, then $\alpha=\psi\circ\gamma:[0,1]\to Z_\Phi$, $\alpha(0)=[x]$, $\alpha(1)=[y]$ and according to Lemma~\ref{FTT4} and \eqref{FTEq7},
$$
\ell(\alpha)=\ell(\phi\circ\alpha)=\ell(\Phi\circ\gamma).
$$	
Therefore, the definition of $d_\Phi$ yields
\begin{equation}
\label{eq20}
d_\Phi([x],[y])=d_\Phi(x,y)=\inf_\gamma\ell(\Phi\circ\gamma)=\inf_\alpha \ell(\alpha).
\end{equation}
We proved that $d_\Phi([x],[y])$ equals the infimum of length of curves $\alpha$ in $Z_\Phi$ that have a special form $\alpha=\psi\circ\gamma$. But the infimum over all curves in $Z_\Phi$ that connect $[x]$ to $[y]$ cannot be smaller than
$d_\Phi([x],[y])$ so $d_\Phi([x],[y])$ is equal to the infimum over all curves in $Z_\Phi$ that connect $[x]$ to $[y]$.

Now suppose that additionally $X$ is compact. Then $Z_\Phi$ is also compact as a $C_qL$-Lipschitz image of $X$ and hence it is geodesic by Corollary~\ref{T2}.
\end{proof}	

In \eqref{eq20} we proved that the distance in $Z_\Phi$ is obtained as the infimum of lengths over a subclass of curves $\alpha=\psi\circ\gamma$ connecting the given two points. While, in general, not every rectifiable curve in $Z_\Phi$ is of that form (see Example~\ref{E1}), all rectifiable curves in $Z_\Phi$ can be well approximated by such curves.
\begin{lemma}
\label{T34}
Let $\alpha:[0,1]\to Z_\Phi$ be a Lipschitz curve, and let $x,y\in X$ be such that $\psi(x)=\alpha(0)$, $\psi(y)=\alpha(1)$. Then there is a sequence of Lipschitz curves
$\gamma_n:[0,1]\to X$, such that 
\begin{itemize}
\item[(a)] $\gamma_n(0)=x$, $\gamma_n(1)=y$,
\item[(b)] $\alpha_n:=\psi\circ\gamma_n$ converges uniformmly to $\alpha$,
\item[(c)] $\ell(\alpha_n)\to\ell(\alpha)$.
\item[(d)] If $\alpha$ is closed, we may choose $\gamma_n$ to be closed by taking $x=y$.
\end{itemize}
\end{lemma}
\begin{example}
\label{E1}
We shall construct an example in which we necessarily have $\ell(\gamma_n)\to\infty$. In such a case, it is not possible to construct a Lipschitz curve $\gamma:[0,1]\to X$ satisfying $\alpha=\psi\circ\gamma$.

Let $X=[0,3]$, $Y=[0,2]$, and let $\Phi:X\to Y$ be defined by
$$
\Phi(x)=
\begin{cases}
x     & \text{if  $0\leq x\leq 1$} \\
1     & \text{if  $1\leq x\leq 2$} \\
x-1   & \text{if  $2\leq x\leq 3$.}
\end{cases}
$$
Then $Z_\Phi=Y$, $\psi=\Phi$, and $\phi=\id$. If $\alpha:[0,1]\to [0,2]=Z_\Phi$ is a Lipschitz curve such that for some $s<t$, $\alpha(s)<1$, $\alpha(t)>1$, then for all sufficiently large $n$, 
$\gamma_n(s)<1$, $\gamma_n(t)>2$, and hence $\ell(\gamma_n|_{[s,t]})>1$. Therefore, if $\alpha$ is a highly oscillating curve that crosses the point $1\in [0,2]=Z_\Phi$ infinitely many times, we necessarily have $\ell(\gamma_n)\to\infty$.
\end{example}
\begin{proof}[Proof of Lemma~\ref{T34}]
Choose a sequence of partitions
$
0=t_{n,0}<t_{n,1}<\ldots<t_{n,k_n}=1,
$
such that
\begin{equation}
\label{eq21}
\ell(\alpha)-\frac{1}{n}\leq 
\sum_{i=0}^{k_n-1}d_\Phi(\alpha(t_{n,i}),\alpha(t_{n,i+1}))\leq\ell(\alpha),
\quad
\Delta_n=\max_i|t_{n,i+1}-t_{n,i}|\stackrel{n\to\infty}{\longrightarrow} 0.
\end{equation}
Fix 
$x_{n,0}=x$, $x_{n,t_k}=y$, and
$x_{n,i}\in X$ satisfying $[x_{n,i}]=\psi(x_{n,i})=\alpha(t_{n,i})$. Since
$$
d_\Phi(\alpha(t_{n,i}),\alpha(t_{n,i+1}))=d_\Phi(x_{n,i},x_{n,i+1}),
$$
the definition of $d_\Phi$ yields the existence of a curve
$$
\gamma_{n,i}:[t_{n,i},t_{n,i+1}]\to X,
\quad
\gamma_{n,i}(t_{n,i})=x_{n,i},
\quad
\gamma_{n,i}(t_{n,i+1})=x_{n,i+1}
$$
such that
\begin{equation}
\label{eq22}
d_\Phi(\alpha(t_{n,i}),\alpha(t_{n,i+1}))\leq
\ell(\Phi\circ\gamma_{n,i}) \leq
d_\Phi(\alpha(t_{n,i}),\alpha(t_{n,i+1}))+\frac{1}{n k_n}\, .
\end{equation}
According to Lemma~\ref{FTT4},
\begin{equation}
\label{eq23}
\ell(\Phi\circ\gamma_{n,i})=\ell(\phi\circ(\psi\circ\gamma_{n,i}))=\ell(\psi\circ\gamma_{n,i}).
\end{equation}
Therefore, \eqref{eq21}, \eqref{eq22}, and \eqref{eq23} yield
\begin{equation}
\label{eq24}
\ell(\alpha)-\frac{1}{n}\leq
\sum_{i=0}^{k_n-1}\ell(\psi\circ\gamma_{n,i})\leq
\ell(\alpha)+\frac{1}{n}.
\end{equation}
For each $n$, define a Lipschitz curve $\gamma_n:[0,1]\to X$ as the concatenation of the curves $\{\gamma_{n,i}\}_{i=0}^{k_n-1}$, and let $\alpha_n=\psi\circ\gamma_n$. 
Clearly, $\gamma_n(0)=x$ and $\gamma_n(1)=y$, which is (a).
If $\alpha$ is closed, we may take $x=y$ and in that case, the curves $\gamma_n$ are closed, which proves (d).

Now, \eqref{eq24} implies that
$\ell(\alpha)-1/n\leq\ell(\alpha_n)\leq\ell(\alpha)+1/n$ which proves that $\ell(\alpha_n)\to\ell(\alpha)$ which is (c).

It remains to prove that $\alpha_n\to\alpha$ uniformly.
Note that
\begin{equation}
\label{eq25}
\alpha_n(t_{n,i})=(\psi\circ\gamma_{n,i})(t_{n,i})=\psi(x_{n,i})=\alpha(t_{n,i}).
\end{equation}

Assume that $\alpha$ is $M$-Lipschitz. Note that
$\alpha_n|_{[t_{n,i},t_{n,i+1}]}=\psi\circ\gamma_{n,i}$,
so \eqref{eq23} and \eqref{eq22} yield
$$
\ell(\alpha_n|_{[t_{n,i},t_{n,i+1}]})\leq d_\Phi(\alpha(t_{n,i}),\alpha(t_{n,i+1}))+\frac{1}{nk_n}\leq 
M|t_{n,i+1}-t_{n,i}|+\frac{1}{nk_n}\, .
$$
Since by \eqref{eq25}, curves $\alpha_n$ and $\alpha$ coincide at the endpoints of the interval $[t_{n,i},t_{n,i+1}]$, for
$t_{n,i}\leq t\leq t_{n,i+1}$ we have
\begin{equation*}
\begin{split}
&d_\Phi(\alpha_n(t),\alpha(t))
\leq
\ell(\alpha|_{[t_{n,i},t_{n,i+1}]}) +
\ell(\alpha_n|_{[t_{n,i},t_{n,i+1}]})\\
& \leq 
M|t_{n,i+1}-t_{n,i}| +\Big(M|t_{n,i+1}-t_{n,i}|+\frac{1}{nk_n}\Big)\leq
2M\Delta_n+\frac{1}{nk_n}\stackrel{n\to\infty}{\longrightarrow} 0.
\end{split}
\end{equation*}
This proves uniform convergence $\alpha_n\to\alpha$ and completes the proof.
\end{proof}

The next result shows that for any $n$, the mapping $\Phi:Z_\Phi\to Y$ preserves the $\H^n$ measure of certain subsets of $Z_\Phi$. 
\begin{theorem}
\label{T27}
Let $\Phi:X\to Y$ be a Lipschitz map between a quasiconvex metric space $X$, and another metric space $Y$. Let $\psi:X\to Z_\Phi$ and $\phi:Z_\Phi\to Y$  be as above.

If $f:\Omega\to X$ is a Lipschitz map defined on an open set $\Omega\subset\bbbr^n$ for some $n$, and $\tilde{X}=f(\Omega)$, then the following holds:
\begin{enumerate}
\item[(a)] $\md(\psi\circ f,x)=\md(\Phi\circ f,x)$ for almost all $x\in\Omega$.
\item[(b)] For any Borel function $g:\psi(\tilde{X})\to [0,\infty]$,
$$
\int_{\psi(\tilde{X})} g(x)\, d\H^n(x)=
\int_{\Phi(\tilde{X})}\Big(\sum_{x\in\phi^{-1}(y)\cap\psi(\tilde{X})} g(x)\Big)\, d\H^n(y).
$$
\end{enumerate}
\end{theorem}
\begin{proof}
The result follows immediately from Theorem~\ref{T18}, because according to Lemma~\ref{FTT4}, the mapping $\phi:Z_\Phi\to Y$ preserves length of curves
$\gamma_{x,v}(t)=(\psi\circ f)(x+tv)$.
\end{proof}

\section{Proof of Theorem~\ref{TM7}}
\label{PTM1}

This section is devoted to the proof of Theorem~\ref{TM7}.
Since the proof will rely on the characterization of $\bbbr$-trees as in part (g) of Lemma~\ref{T33}, before we proceed to the proof of the theorem, we need a few simple lammata about the integral expression there, but first we will start with an informal heuristic discussion.

If $D\subset\bbbr^2$ is a smooth bounded simply connected domain, and $\gamma:\bbbs^1\to\partial D$ is an orientation preserving parametrization of the boundary, then the area of $D$ can be expressed as
$$
A(D)=
\int_D dx\wedge dy=
\int_{\partial D}x\, dy=
\int_{\bbbs^1}\gamma^*(x\, dy).
$$
If $\gamma:\bbbs^1\to\bbbr^2$ is any Lipschitz curve, then
$$
A(\gamma):=\int_{\bbbs^1}\gamma^*(x\, dy)
$$
represents the {\em oriented area enclosed by $\gamma$}---the sum (possibly infinite) of areas of bounded connected components of $\bbbr^2\setminus\gamma(\bbbs^1)$, multipled by the corresponding winding numbers.

Thus, roughly speaking, condition (g) in Lemma~\ref{T33} says that there are no closed curves in $X$ with non-trivial ``holes'', as otherwise we could project such a curve to $\bbbr^2$, by composing it with a suitably constructed Lipschitz map $\pi:X\to\bbbr^2$, so that the resulting curve $\pi\circ\gamma$ would bound a non-zero oriented area
$$
A(\pi\circ\gamma)=\int_{\bbbs^1}(\pi\circ\gamma)^*(x\, dy)\neq 0.
$$
This interpretation is consistent with our intuition that $\bbbr$-trees are geodesic spaces without non-trivial loops.

In fact, we will not use the above geometric interpretation of $A(\gamma)$, but it is important to keep it in mind to provide intuition for the estimates that we do next.

\begin{lemma}
\label{tt1}
Let $\gamma \colon \bbbs^1 \to \bbbr^2$ be a Lipschitz curve and let $\Gamma=(\Gamma_1,\Gamma_2)\colon \bar{\bbbb}^2 \to \bbbr^2$ be any Lipschitz extension of $\gamma$ to the closed unit disk. Then
\begin{equation}
    \label{saz2}
 A(\gamma) = \int_{\bbbb^2} d\Gamma_1\wedge d\Gamma_2=\int_{\bbbb^2}\det D\Gamma.
\end{equation}
\end{lemma}
If $\Gamma$ is smooth up to the boundary, it follows from Stokes' theorem, and in the Lipschitz case one can prove it by using a smooth approximation.
For a detailed proof of a more general result, see for example, \cite[Lemma~4.9]{DHLT}.

If $\gamma:\bbbs^1\to\bbbr^2$ is a Lipschitz curve and $\bar{\gamma}=(\bar{\gamma}_1,\bar{\gamma}_2):[0,1]\to\bbbr^2$, $\bar{\gamma}(t)=\gamma(\exp(2\pi it))$, then $\bar{\gamma}(0)=\bar{\gamma}(1)$, and
$$
A(\gamma)=\int_{\bbbs^1}\gamma^*(x\,dy)=
\int_0^1\bar{\gamma}_1(t)\bar{\gamma}_2'(t)\, dt,
$$
so with a slight abuse of notation we can identify $\gamma$ with $\bar{\gamma}$ and, whenever it is convenient, regard $\gamma$ as a curve $\gamma:[0,1]\to\bbbr^2$ satisfying $\gamma(0)=\gamma(1)$.


\begin{lemma}
\label{ma1}
Suppose that $\gamma_n \colon \bbbs^1 \to \bbbr^2$, $n=1,2,\ldots$ are Lipschitz curves that uniformly converge to a Lipschitz curve $\gamma \colon \bbbs^1 \to \bbbr^2$. Assume also that there exists  $M>0$ such that $\ell(\gamma_n) \leq M$ for all $n$. Then,
$$
\lim_{n\to \infty} A(\gamma_n) = A(\gamma) \, .
$$
\end{lemma}
\begin{proof}
If we write $\gamma_n=(\gamma_{n,1},\gamma_{n,2})$, then
\begin{align*}
    A(\gamma) - A(\gamma_n) &= \int_{0}^{1} \gamma_1  \gamma_2' \, dt - \int_{0}^{1} \gamma_{n,1}  \gamma_{n,2}' \, dt \\
    &= \int_{0}^{1} (\gamma_1-\gamma_{n,1})  \gamma_2' \, dt + \int_{0}^{1} \gamma_{n,1}(\gamma_2' - \gamma_{n,2}')\, dt\\
    &= \int_{0}^{1} (\gamma_1-\gamma_{n,1})  \gamma_2' \, dt - \int_{0}^{1} \gamma_{n,1}'(\gamma_2 - \gamma_{n,2})\, dt\, ,
\end{align*}
where the last equality follows from the integration by parts. If we show that the the right-hand side converges to zero we are done.
Since $|\gamma'|$ is bounded and $\gamma_{n,1}\to\gamma_1$ uniformly, it follows that
$$
\left \vert \int_{0}^{1} (\gamma_1-\gamma_{n,1})  \gamma_2' \, dt \right \vert \leq \Vert\gamma'\Vert_\infty \int_0^{1} |\gamma_1-\gamma_{n,1}| \, dt
\stackrel{n\to\infty}{\longrightarrow} 0.
$$
Next, we have
$$
\left \vert \int_{0}^{1} \gamma_{n,1}'(\gamma_2 - \gamma_{n,2})\, dt\right \vert \leq
\|\gamma_2 - \gamma_{n,2}\|_\infty \int_{0}^{1} |\gamma_{n,1}'| \, dt\leq
M\|\gamma_2 - \gamma_{n,2}\|_\infty\stackrel{n\to\infty}{\longrightarrow} 0.
$$
The proof is complete.
\end{proof}

\begin{proof}[Proof of Theorem~\ref{TM7}]
The implication ($\Rightarrow$) easily follows from Theorem~\ref{T17} and from Lemma~\ref{T32}. 
Therefore, it remains to prove the implication ($\Leftarrow$).
Some ideas used below are based on the proof of Theorem~5 in \cite{wengery}.

Let $f\colon [0,1]^n \to X$ be Lipschitz, and such that $\rank\md(f,x)\leq 1$ a.e.
We need to prove that $f$ factors through an $\bbbr$-tree.
Since $f$ factors through the space $Z:=Z_f$ constructed as in Section~\ref{FLM}, it suffices to prove that $Z$ is an $\bbbr$-tree.

Let $\psi$ and $\phi$ be the Lipschitz maps $\psi\colon [0,1]^n \to Z$ and $\phi \colon Z \to X$, constructed in Section~\ref{FLM}, so $f = \phi \circ \psi$.  


By Corollary~\ref{FTT5}, $Z$ is geodesic, so according to part (g) of Lemma~\ref{T33}, it suffices to show that  for every Lipschitz curve $\alpha \colon \bbbs^1 \to Z$ and every Lipschitz function $\pi\colon  Z \to \mathbb{R}^2$,
\begin{equation}
\label{elm1}
A(\pi \circ \alpha)=\int_{\bbbs^1}(\pi\circ\alpha)^*(x\, dy)=0.
\end{equation}
First, assume that $\alpha=\psi\circ\gamma:\bbbs^1\to Z$, where $\gamma:\bbbs^1\to [0,1]^n$ is a Lipschitz curve. Let $g:\bar{\bbbb}^2\to [0,1]^n$ be a Lipschitz extension of $\gamma$ to the closed unit disc. 

For technical reasons that will be explained later, we need an extension $g$ with the property that it maps the interior of the disc to the interior of the cube $g(\bbbb^2)\subset (0,1)^n$. That however, can be easily guaranteed. Indeed, if $g:\bar{\bbbb}^2\to [0,1]^n$ is any Lipschitz extension of $\alpha$, then 
$\tilde{g}:\bar{\bbbb}^2\to [0,1]^n$ defined by
$$
\tilde{g}(x)=
\big(g(x)-(1/2,1/2,\ldots,1/2)\big)|x|+
(1/2,1/2,\ldots,1/2),
$$
agrees with $g$ on the boundary of the disc, $|x|=1$ (and hence $\tilde{g}$ is an extension of $\alpha$), and when $|x|<1$, $\tilde{g}(x)$ is in the interior of the segment connecting the center of the cube $(1/2,1/2,\ldots,1/2)$ to $g(x)$ so $\tilde{g}(x)$ belongs to the interior $(0,1)^n$ of the cube.

Then, $\pi\circ\psi\circ g:\bar{\bbbb}^2\to\bbbr^2$ is a Lipschitz extension of $\pi\circ\psi\circ\gamma:\bbbs^1\to\bbbr^2$, and Lemma~\ref{tt1} yields
\begin{equation}
\label{eq19}
A(\pi \circ \alpha) = A(\pi\circ\psi\circ\gamma)=\int_{\bbbb^2} \det D(\pi \circ \psi \circ g)\, .
\end{equation}
Clearly, for Lipschitz mappings $h:\bbbr^n\supset\Omega\to\bbbr^m$, $\rank Dh(x)=\rank\md(h,x)$, whenever $h$ is differentiable at $x$, so Lemma~\ref{T32},
and part (a) of Theorem~\ref{T27} give
\begin{equation}
\label{eq18}
\rank D(\pi\circ\psi\circ g)=
\rank \md(\pi\circ\psi\circ g)\leq \rank \md(\psi\circ g)=\rank \md(f\circ g)\
\text{a.e.}
\end{equation}
Since by assumptions, $\rank \md(f) \leq 1$ a.e., Proposition~\ref{T28} implies that
$\rank \md(f \circ g) \leq 1$ a.e. 
(this is where we use the assumption that $g(\bbbb^2)\subset (0,1)^n$),
and hence 
$\rank D(\pi\circ\psi\circ g)\leq 1$ by \eqref{eq18}. This and \eqref{eq19} proves that $A(\pi\circ\alpha)=0$.

That is, we proved \eqref{elm1} for curves $\alpha:\bbbs^1\to Z$ that factor through $[0,1]^n$, $\alpha=\psi\circ\gamma$, while we need to prove \eqref{elm1} for all Lipschitz curves $\alpha:\bbbs^1\to Z$. 
We can however, easily pass to the general case with the help of Lemmata~\ref{T34} and~\ref{ma1}.

Let $\alpha:\bbbs^1\to Z$ be a Lipschitz curve. Let $\alpha_j=\psi\circ\gamma_j$ be the approximation described in Lemma~\ref{T34}. Note that by Lemma~\ref{T34}(d) we may choose the curves $\gamma_j$ are closed.

Then $A(\pi\circ\alpha_j)=0$, because of the special form of $\alpha_j$. Since $\alpha_j\to\alpha$ uniformly, and the curves $\alpha_j$ have uniformly bounded length, the curves $\pi\circ\alpha_j:\bbbs^1\to\bbbr^2$ have uniformly bounded length and they converge uniformly to $\pi\circ\alpha$, so Lemma~\ref{ma1} yields
$$
0=A(\pi\circ\alpha_j)\to A(\pi\circ\alpha)
$$
and \eqref{elm1} follows. The proof is complete.
\end{proof}

\section{Proof of Theorem~\ref{TM1B}}
\label{PTM3}

\begin{proof}[Proof of Theorem~\ref{TM1B}]
We will prove the theorem by proving implications
$(b')\Rightarrow(c')\Rightarrow(d')\Rightarrow(e')\Rightarrow(f')\Rightarrow(b')$.

\noindent
\underline{Implication $(b')\Rightarrow(c')$}.
According to Proposition~\ref{Tl14}, almost all $x \in E\cap (0,1)^{n+m}$ have the property that for any $j \in \bbbn$, and for all sufficiently small $r>0$, the set $f(B(x,r))$ can be covered by $j^{k}$ balls each of radius 
$3\sqrt{k}Lr/j$, where $L=\lip(f)$ and $k=\rank\md(f,x)\leq n-1$. Since 
$3\sqrt{k}Lr/j\leq 3\sqrt{n}Lr/j$, for all sufficiently small $r>0$, we have
$$
\H^n_\infty (f(B(x,r))) \lesssim_{n,L} j^{k}\Big(\frac{r}{j}\Big)^n,
\quad
\text{so}
\quad
\Theta^{*n}(f,x) \lesssim_{n,L} \frac{1}{j^{n-k}}\, .
$$
Since $n-k\geq n- (n-1) = 1,$ and $j$ was arbitrary, we have $\Theta^{*n}(f,x)=0$. Since, this is true for almost all $x\in E\cap (0,1)^{n+m}$, $(c')$ is proved.

\noindent
\underline{Implication $(c')\Rightarrow(d')$} is obvious.

\noindent
\underline{Implication $(d')\Rightarrow(e')$} is a direct consequence of the following estimate.
\begin{proposition}
\label{T37}
If $f:Q_o=[0,1]^{n+m}\to X$, $n\geq 1$, $m\geq 0$, is a Lipschitz map into a metric space and $E\subset Q_o$ is a measurable set, then
$$
\H^{n,m}_\infty(f,E) \lesssim_{n,m} \int_E \Theta_*^n(f,x) \, d\H^{n+m}(x)\, .
$$
\end{proposition}
\begin{remark}
This is a slight improvement of \cite[Proposition~5.1]{HZ}. The proof presented below is similar to the one in \cite{HZ}, but Proposition~5.1 in \cite{HZ} involved a slightly different definition of $\H^{n,m}_\infty$ than the one we use now, and this is one of the reasons for providing details.
\end{remark}
\begin{proof}

The function $\Theta_*^n(f,\cdot)$ is integrable as bounded and measurable. Let
$$
A=\{x\in E\cap (0,1)^{n+m}:\, \text{$x$ is a Lebesgue point of $\Theta_*^n(f,\cdot)$}\}.
$$
Fix $\eps>0$. There is an open set $U\subset Q_o$, such that $A\subset U$ and
$$
\int_U\Theta_*^n(f,x)\, d\H^{n+m}(x)<
\int_A\Theta_*^n(f,x)\, d\H^{n+m}(x)+\eps=
\int_E\Theta_*^n(f,x)\, d\H^{n+m}(x)+\eps.
$$
Let $x\in  A$. By the definition of $\Theta^n_*(f,x)$, there is a sequence $r_x^i\searrow 0$ such that
$$
\frac{\H^n_\infty(f(B(x,r_x^i)))}{\omega_n(r_x^i)^n}<
\Theta^n_*(f,x)+\eps,
\quad
B(x,r_x^i)\subset U.
$$
For each $i$, we can find a closed dyadic cube $Q_x^i$ such that
$$
x\in Q_x^i\subset B(x,r_x^i)
\quad
\text{and}
\quad
r_x^i\lesssim_{n,m}\diam Q_x^i,
$$
so
\begin{equation}
\label{eq26}
\frac{\H^n_\infty(f(Q_x^i))}{\omega_n(r_x^i)^n}<\Theta_*^n(f,x)+\eps,
\quad
Q_x^i\subset U.
\end{equation}
Since averages of $\Theta_*^n(f,\cdot)$ over the cubes $Q_x^i$ converge to $\Theta_*^n(f,x)$, as $i\to \infty$, by assuming that all $r_x^i$ are sufficiently small, we may guarantee that
\begin{equation}
\label{eq27}
\Theta_*^n(f,x)<
\mvint_{Q_x^i}\Theta_*^n(f,y)\, d\H^{n+m}(y)+\eps.
\end{equation}
Hence \eqref{eq26} and \eqref{eq27} yield
$$
\H^n_\infty(f(Q_x^i))(\diam Q_x^i)^m\lesssim_{n,m}
\int_{Q_x^i}\Theta_*^n(f,y)\, d\H^{n+m}(y)+2\eps(\diam Q_x^i)^{n+m}.
$$
The collection of (closed) dyadic cubes
$$
\mathcal{Q}=\{Q_x^i:\, x\in A,\ i\in\bbbn\}
$$
forms a covering of $A$. Dyadic cubes have an important property that given two dyadic cubes, they have disjoint interiors or one is contained in another one. Thus leaving in $\mathcal{Q}$ only the largest cubes that are not contained in any larger cube from $\mathcal{Q}$, we obtain a subfamily $\{ Q_j\}_j\subset\mathcal{Q}$ of dyadic cubes with pairwise disjoint interiors such that $A\subset\bigcup_j Q_j$.
Since $\H^{n+m}(E\setminus A)=0$, the definition of $\H^{n,m}_\infty$ along with \eqref{eq28} yield
\begin{equation*}
\begin{split}
&\H^{n,m}_\infty(f,E)=\H^{n,m}_\infty(f,A)
\leq 
\sum_j \H_\infty^n(f(Q_j))(\diam Q_j)^m\\
&\lesssim_{n,m}
\sum_j \int_{Q_j}\Theta_*^n(f,y)\, d\H^{n+m}(y) +
2\eps\sum_j (\diam Q_j)^{n+m}\\
&\lesssim_{n,m}
\int_U \Theta_*^n(f,y)\, d\H^{n+m}(y) +2\eps
\leq
\int_E \Theta_*^n(f,y)\, d\H^{n+m}(y) +3\eps
\end{split}    
\end{equation*}
and the result follows upon letting $\eps\to 0$.
\end{proof}

\noindent
\underline{Implication $(e')\Rightarrow(f')$} is obvious because of \eqref{elm6}.

\noindent\underline{Implication $(f')\Rightarrow(b')$.}
We will use here notation and facts from Section~\ref{HHC}.
According to Lemmata~\ref{T5} and~\ref{T7} it suffices to prove the following result.
\begin{proposition}
\label{T26}
If $f=(f_1,f_2,\ldots)\colon Q_o=[0,1]^{n+m}\to \ell^\infty$, $n\geq 1$, $m\geq 0$, is Lipschitz, $E\subset Q_o$ is measurable, and $\hat{\H}^{n,m}_\infty(f,E)=0$, then $\rank Df(x)\leq n-1$, for $\H^{n+m}$-almost all $x\in E$.
\end{proposition}
\begin{proof}

We shall ignore the points on the boundary of the cube. If $\rank Df(x) \geq n$ at some $x$, then there exist indices $i_1<i_2<\ldots<i_n$ such that $\nabla f_{i_1}(x), \ldots, \nabla f_{i_n}(x)$ are linearly independent. If we let $\pi:\ell^\infty\to\bbbr^n$ be the projection
$\pi(x)=(x_{i_1},\ldots,x_{i_n})$, then the latter statement is equivalent to $\rank D(\pi\circ f)(x)=n$. There are countably many possibilities for choosing $n$ natural numbers, so, the proposition will follow once we prove that for any $i_1<i_2<\ldots<i_n$, the equality $\rank D(\pi\circ f)(x)=n$ occurs only on a null subset of $E$, i.e.\ 
$\rank D(\pi\circ f)(x)\leq n-1$, for $\H^{n+m}$-almost all $x\in E$.

But this follows immediately from the next lemma since by easy estimates\footnote{Note that $\pi$ is a Lipschitz map with the Lipschitz constant $\sqrt{n}$.}
$$
\hat{\H}^{n,m}_\infty(\pi\circ f,{E})
\leq (\lip\pi)^n\hat{\H}^{n,m}_\infty(f,E)=0.
$$
We need to apply the next lemma to $F=\pi\circ f:Q_o\to\bbbr^n$.
Thus the next lemma is the last missing piece in the proof of the implication.
\begin{lemma}
\label{T22}
If $F:Q_o=[0,1]^{n+m}\to\bbbr^{n}$, $n\geq 1$, $m\geq 0$ is Lipschitz, $E\subset Q_o$ is measurable, and 
$\hat{\H}^{n,m}_\infty(F,E)=0$, then $\rank DF(x)\leq n-1$ for $\H^{n+m}$-almost all $x\in E$.
\end{lemma}
\begin{proof}
According to the classical co-area formula (Lemma~\ref{T31})
$$
\int_E |J_F(x)|\, d\H^{n+m}(x)=\int_{\bbbr^n}\H^m(F^{-1}(y)\cap E)\, d\H^n(y).
$$
Note that if $F$ is differentiable at $x$, then $|J_F(x)|=0$ if and only if $\rank DF(x)\leq n-1$. Therefore, it suffices to show that $|J_F(x)|=0$ for almost all $x\in E$. To this end, it suffices to show that $\H^m(F^{-1}(y)\cap E)=0$ for
$\H^n$-almost all $y\in\bbbr^n$.

Assume that $m\geq 1$. A similar argument works for $m=0$.

Let $\eps>0$ be given.
Since $\hat{\H}^{n,m}_\infty(F,E)=0$, it follows that there is a covering $E\subset\bigcup_{i} A_i\subset Q_o$, such that
$$
\sum_{i}\H^n_\infty(F(A_i))(\diam A_i)^m<\eps.
$$
According to Lemmata~\ref{T35} and~\ref{T36}, there are Borel sets $Z_i\subset\bbbr^n$ such that $F(A_i)\subset Z_i$ and $\H^n_\infty(F(A_i))=\H^n(F(A_i))=\H^n(Z_i)$.
We have
\begin{equation*}
\begin{split}
\H^m_\infty (F^{-1}(y)\cap E)
&\lesssim_m 
\sum_{i} (\diam (F^{-1}(y)\cap A_i))^m
=
\sum_i (\diam (F^{-1}(y)\cap A_i))^m \chi_{F(A_i)}(y)\\
&\leq 
\sum_i (\diam A_i)^m\chi_{Z_i}(y).
\end{split}    
\end{equation*}
Since the function on the right hand side is Borel, integration yields
$$
\int_{\bbbr^n}^* \H^m_\infty(F^{-1}(y)\cap E)\, d\H^n(y)
\lesssim_m
\sum_i (\diam A_i)^m \H^n(Z_i)
=
\sum_i (\diam A_i)^m \H^n_\infty (F(A_i))<\eps.
$$
Since this is true for any $\eps>0$, we conclude that
$$
\int_{\bbbr^n}^* \H^m_\infty(F^{-1}(y)\cap E)\, d\H^n(y)=0.
$$
Thus, $\H^m_\infty(F^{-1}(y)\cap E)=0$, and hence  $\H^m(F^{-1}(y)\cap E)=0$, for $\H^n$-almost all $y\in\bbbr^n$. The proof of Lemma~\ref{T22} is complete.
\end{proof}
Lemma~\ref{T22} completes the proof of Proposition~\ref{T26} and hence that proof of the implication.
\end{proof}
This was the last implication to prove and therefore, the proof of Theorem~\ref{TM1B} is complete.
\end{proof}

\end{document}